\documentclass{amsart}
\usepackage{amsmath,amsthm}
\usepackage {latexsym}
\usepackage{amssymb}

\newcommand{\beal}{\begin{align}}
\newcommand{\enal}{\end{align}}
\newcommand{\bealn}{\begin{align*}}
\newcommand{\enaln}{\end{align*}}
\newcommand{\bear}{\begin{eqnarray}}
\newcommand{\eear}{\end{eqnarray}}
\newcommand{\beeq}{\begin{equation}}
\newcommand{\eneq}{\end{equation}}

\newcommand{\eps}{{\varepsilon}}
\newcommand{\R}{{\mathbb R}}

\newcommand{\la}{\langle}
\newcommand{\ra}{\rangle}

\def\bm{\left[ \begin{array}{cc}}
\def\endm{\end{array}\right]}

\def\eps{\varepsilon}

\def\bm{\left[\begin{matrix} }
\def\endm{\end{matrix}\right]}
\def\la{\langle}
\def\ra{\rangle}

\def\les{\lesssim}

\def\R{{\mathbb R}}

\newtheorem{theorem}{Theorem}
\newtheorem{lemma}[theorem]{Lemma}

\newtheorem{prop}[theorem]{Proposition}

\theoremstyle{remark}

\renewcommand{\hat}{\widehat}
\renewcommand{\epsilon}{\eps}
\renewcommand{\tilde}{\widetilde}
\numberwithin{equation}{section}
\numberwithin{theorem}{section}

\begin{document}

\title{On stability of the catenoid under vanishing mean curvature flow on Minkowski space.}

\author[J. Krieger]{Joachim Krieger}
\address{(JK) Dept. of Mathematics, Station 8, CH-1015 Lausanne, Switzerland.}
\email{joachim.krieger@epfl.ch}
\urladdr{http://pde.epfl.ch}

\author[H. Lindblad]{Hans Lindblad}
\address{(HL) Dept. of Mathematics, Johns Hopkins University}
\email{lindblad@math.jhu.edu}
\urladdr{http://www.math.jhu.edu/~lindblad/}

\begin{abstract} We establish basic local existence as well as a stability result concerning small perturbations of the Catenoid minimal surface in $\R^3$ under hyperbolic vanishing mean curvature flow.
\end{abstract}

\maketitle

\section{Introduction}

The minimal surface equation in Riemannian geometry has a natural analogue on a Lorentzian background. In particular, working on a Minkowski background $\R^{2+1} = \{(t, x)| x\in \R^2\}$ equipped with the standard metric $dg = dt^2 - \sum_{i=1,2}dx_i^2$ and considering surfaces $S$ which for fixed $t$ are graphs of functions $\phi(t, x)$ over $\R^2$, we find the equation
\begin{equation}\label{eq:HypMin}
\frac{\partial}{\partial t}\big[\frac{\phi_t}{\sqrt{1+|\nabla_x \phi|^2 - \phi_t^2}}\big] - \sum_{i=1,2}\frac{\partial}{\partial x_i}\big[\frac{\phi_{x_i}}{\sqrt{1+|\nabla_x \phi|^2 - \phi_t^2}}\big]  = 0
\end{equation}
We note that this equation appears in string theory \cite{Ho}.
\\
As of this point in time, there appears to be  no general theory for dealing with quasilinear problems of this nature, and even perturbative questions appear highly challenging. The most basic of these is to study the stability of the trivial solution $\phi = 0$ describing a plane, which was effected in \cite{Bren}, \cite{Lin}. We are not aware of works studying the stability under \eqref{eq:HypMin} of other minimal surfaces in $\R^{3}$.
Here we would like to initiate the study of the (in)stability of another natural static solution (i. e. minimal surface in the Riemannian sense) of \eqref{eq:HypMin}, the {\it{Catenoid}}.
This is the solution given by the graph of 
\begin{equation}\label{eq:cat}
\phi(t, r) = Q(r) : = \log (r+\sqrt{r^2 - 1}),\,r = |x|\geq 1
\end{equation}
In order to obtain some basic idea of what to expect, it is natural to look at elliptic and parabolic analogues of \eqref{eq:HypMin}, and in particular, the question of stability of this solution in the variational sense. Here it has been well-known since the 1980's \cite{SchF} that the Catenoid (as well as all other non-planar minimal surfaces) are {\it{unstable}}, and thus at least for the parabolic analogue of \eqref{eq:HypMin} generic perturbations of \eqref{eq:cat} are expected to lead to singularity formation (via neck pinching) and the formation of two planes. It is not too far-fetched to surmise that the solution \eqref{eq:cat} is also unstable for the flow \eqref{eq:HypMin}, although we are far from having an argument for this.
In the following sections, we aim to settle some very basic questions concerning \eqref{eq:HypMin}: first, the most basic issue is that of understanding local well-posedness for arbitrary (sufficiently smooth) perturbations of \eqref{eq:cat}. Second, in order to better understand potential singularity formation for generic perturbations of \eqref{eq:cat}, we establish a result on stability of  \eqref{eq:cat} for certain generic radial perturbations which are supported far away from the collar $r = 1$, as long as the resulting deformation stays away from the collar (here we take advantage of the Huyghen's principle). This result implies in particular that for these solutions, a singularity can only set in once the 'collar starts to move'.

\section{Local existence} Instead of working with an explicit graph representation which yields the description \eqref{eq:HypMin}, one may also work with an implicit description. Then the minimal surface equation for a hypersurface
$$
\Psi=0
$$ in
Minkowski space-time is given by
$$
\nabla_\alpha \bold{N}^\alpha\Big|_{\Psi=0}=0,\qquad \bold{N}^\alpha=\nabla^\alpha
\Psi/|\nabla\Psi|,\qquad \nabla^\alpha=m^{\alpha\beta}\nabla_\beta
$$
and $\nabla_\alpha=\partial_\alpha$, $|W|=\sqrt{m(W,W)}$, where $m$ is the Minkowski metric
$diag(-1,1,1,1)$. This can also be written
$$
\big(|\nabla\Psi|^2 m^{\alpha\beta} -\nabla^\alpha \Psi\nabla^\beta\Psi\big) \nabla_\alpha\nabla_\beta\Psi=0
$$

\subsection{Hyperbolicity}
We write a four vector $X=(X^0,X^\prime)$, where $X^\prime$ is a three vector.
For a three vector let $|X^\prime|$ denote the Euclidean distance. Set
$$
g^{\alpha\beta}(X)=|\hat{X}|^2 m^{\alpha\beta} -\hat{X}^\alpha \hat{X}^\beta, \qquad \text{where}\quad
\hat{X}=X/|X^\prime|
$$
With repeated upper and lower Greek indices $\alpha,\beta,\gamma,\delta,...$ being summed over $0,1,2,3$ and repeated Latin indices $i,j,...$ being summed over $1,2,3$ only we have:
\begin{lemma} We have
\begin{equation}
g^{\alpha\beta}(X)\xi_\alpha\xi_\beta=-(\xi_0+T^j\xi_j)^2+\gamma^{ij}\xi_i\xi_j
\end{equation}
where
\begin{equation}
T^j=\hat{X}^0 \hat{X}^j,\qquad\text{and}\qquad
\gamma^{ij}=(1-(\hat{X}^0)^2)\big(\delta^{ij}-\hat{X}^i\hat{X}^j \big).
\end{equation}
and $\hat{X}=X/|X^\prime|$, where $|X^\prime|=\sqrt{X_1^2+X_2^2+X_3^2}$.
We have
\begin{equation}
(1-|T|^2)\big(|n|^2-(T^k n_k/|T|)^2\big) \leq \gamma^{ij} n_i n_j\leq (|n|-|T^k n_k|)^2
\end{equation}
\end{lemma}
\begin{proof} Completing the square we get
\begin{multline*}
g^{\alpha\beta}(X)\xi_\alpha\xi_\beta= -\big(\xi_0^2 +2\hat{X}^0\hat{X}^j\xi_0\xi_j +(\hat{X}^0)^2\hat{X}^i \hat{X}^j\xi_i\xi_j\big) +|\hat{X}|^2 \big(\delta^{ij}\xi_i\xi_j
-\hat{X}^i \hat{X}^j\xi_i\xi_j\big)\\
g^{\alpha\beta}(X)\xi_\alpha\xi_\beta= -\big(\xi_0 +\hat{X}^0\hat{X}^j\xi_j\big)^2  +|\hat{X}|^2 \big(\delta^{ij}
-\hat{X}^i \hat{X}^j\big)\xi_i\xi_j
\end{multline*}
If $n$ is a unit vector then one sees that
\begin{equation}
\gamma^{ij} n_i n_j=(1-(\hat{X}^0)^2)(1-(\hat{X}^k n_k)^2)=(1-|\hat{X}^0| |\hat{X}^k n_k|)^2
-(|\hat{X}^0|-|\hat{X}^k n_k|)^2
\end{equation}
from which the last inequality follows.
\end{proof}

Returning to the graph representation by writing $\Psi(t,x,y,z)=z-\phi(t,x,y)$, we have $X=(\phi_t,-\phi_x,-\phi_y,1)$ and $g^{\alpha\beta}\partial_\alpha\partial_\beta\Psi=0$ becomes
\begin{equation}
g^{\alpha\beta}(\partial\phi)\partial_\alpha\partial_\beta\phi=0
\end{equation}
where the sum is only over $\alpha,\beta=0,1,2$, since $\phi$ is independent of $z$; this is seen to co-incide with \eqref{eq:HypMin}. The symbol for this operator
is the same as \eqref{eq:Symbol} but with the sum over only $\alpha,\beta=0,1,2$, i.e. with $\xi$ replaced by $(\xi_0,\xi_1,\xi_2,0)$.
This satisfies G{\aa}rding's hyperbolicity condition, see \cite{J}, if $\sum_{i,j=1,2}\gamma^{ij}\xi_i\xi_j$ is positive definite,
which is the case if the initial surface is time like:
\begin{equation}\label{eq:spacelike}
|X|^2=|\nabla\Psi|^2=1+\phi_x^2+\phi_y^2-\phi_t^2>0,
\end{equation}
since $(\hat{X}^{\prime 1})^2+(\hat{X}^{\prime 2})^2=(\phi_x^2+\phi_y^2)/(\phi_x^2+\phi_y^2+1)<1$.

\subsection{Energy Estimates}

\begin{lemma} Let $g^{\alpha\beta}$ be as in the previous lemma and suppose that $|\hat{X}^0|\leq 1-\varepsilon$.
Suppose the $\phi$ solves the equation
\begin{equation}
\sum_{\alpha\beta=0,1,2} g^{\alpha\beta}(X)\partial_\alpha\partial_\beta \phi=G
\end{equation}
in a set
$$
D_T=\{(x,t); \, |x-x_0|<R-t,\quad 0\leq t<T\}
$$
and let $S_t=\{(x,t); |x-x_0|<R-t\}$ and
$$
E(t)=\int_{S_t} \big((\partial_0+T^j\partial_j)\phi\big)^2+\gamma^{ij}\partial_i\phi\, \partial_j\phi \,\, dx
$$
Then
$$
\sqrt{E(t)}\leq  e^{\int_0^t C_\varepsilon n(s)\, ds} \Big(\sqrt{E(0)}+\int_0^T  \|G(t,\cdot)\|_{L^2(S_t)}\, dt \Big),\qquad n(t)=\sup_{S_t}{\big(|\gamma^\prime|+|T^\prime|\big)}
$$
has energy estimates, with a constant depending on $\varepsilon$ and $T$ and some norms of $X$.
\end{lemma}
\begin{proof} We have
\begin{equation*}\label{eq:Symbol}
G=g^{\alpha\beta}\partial_\alpha\partial_\beta\phi= -\big(\partial_0 +T^j\partial_j\big)^2 \phi
 +\partial_i \big(\gamma^{ij}\partial_j\phi\big) \\
+\big(\big(\partial_0 +T^k\partial_k\big)T^j
-\partial_i \gamma^{ij} \big)\partial_j\phi
\end{equation*}
Note that
\begin{multline}
(\partial_0 +T^k \partial_k ) \big(\gamma^{ij}\partial_i\phi\, \partial_j\phi\big)\\
 =
2\gamma^{ij}\partial_i\phi\, \partial_j (\partial_0 +T^k \partial_k )\phi
-2\gamma^{ij}(\partial_j T^k)\, \partial_i \phi\, \partial_k\phi+\big((\partial_0 +T^k \partial_k ) \gamma^{ij}\big)\partial_i\phi\, \partial_j\phi\\
=2\partial_j\Big(\gamma^{ij}\partial_i\phi\,  (\partial_0 +T^k \partial_k )\phi\Big)
-2(\partial_0 +T^k \partial_k )\phi\, \partial_j\big(\gamma^{ij}\partial_i\phi\big)\\
+\big((\partial_0 +T^k \partial_k ) \gamma^{ij}\big)\partial_i\phi\, \partial_j\phi-2\gamma^{ij}(\partial_j T^k)\, \partial_i \phi\, \partial_k\phi
\end{multline}
We hence get
\begin{multline*}
(\partial_0 +T^k \partial_k )\Big(\big((\partial_0 +T^j \partial_j )\phi\big)^2
+\gamma^{ij}\partial_i \phi\,\partial_j \phi \Big)\\
=2 \big((\partial_0 +T^k \partial_k )\phi\big)
(\partial_0 +T^j \partial_j )^2\phi+(\partial_0 +T^k \partial_k ) \big(\gamma^{ij}\partial_i\phi\, \partial_j\phi\big)\\
=2 \big((\partial_0 +T^k \partial_k )\phi\big)
\Big((\partial_0 +T^j \partial_j )^2\phi-\, \partial_j\big(\gamma^{ij}\partial_i\phi\big)\Big)
+2\partial_j\Big(\gamma^{ij}\partial_i\phi\,  (\partial_0 +T^k \partial_k )\phi\Big)
\\
+\big((\partial_0 +T^k \partial_k ) \gamma^{ij}\big)\partial_i\phi\, \partial_j\phi-2\gamma^{ij}(\partial_j T^k)\, \partial_i \phi\, \partial_k\phi\\
=2\partial_j\Big(\gamma^{ij}\partial_i\phi\,  (\partial_0 +T^k \partial_k )\phi\Big)
+2 \big((\partial_0 +T^k \partial_k )\phi\big)\Big(\big(\big(\partial_0 +T^k\partial_k\big)T^j
-\partial_i \gamma^{ij} \big)\partial_j\phi\Big)\\
-2 \big((\partial_0 +T^k \partial_k )\phi\big) G
+\big((\partial_0 +T^k \partial_k ) \gamma^{ij}\big)\partial_i\phi\, \partial_j\phi-2\gamma^{ij}(\partial_j T^k)\, \partial_i \phi\, \partial_k\phi
\end{multline*}
If we integrate this over $S_t$ we get
\begin{equation}
E(t)-E(0)+H(t)=R(t)+G(t),
\end{equation}
Were the flux is given by
\begin{equation}
H(t)=\int_{C_t}H^{\alpha\beta}\partial_\alpha\phi\, \partial_\beta \phi \, dx
\end{equation}
where $C_T=\{ (x,t);\, |x-x_0|=R-t, \, 0\leq t\leq T\}$ and
\begin{equation*}
H^{\alpha\beta} \partial_\alpha \phi\, \partial_\beta\phi=
\Big(\big((\partial_0 +T^j \partial_j )\phi\big)^2
+\gamma^{ij}\partial_i \phi\,\partial_j \phi \Big) (1+T^k n_k) -2 n_j\gamma^{ij}\partial_i\phi\,  (\partial_0 +T^k \partial_k )\phi,
\end{equation*}
and the remainder is
\begin{equation}
R(t)=\int_0^t\int_{D_s} R^{\alpha\beta}\partial_\alpha\phi\, \partial_\beta \phi \,\, dx ds
\end{equation}
where
\begin{multline}
R^{\alpha\beta}\partial_\alpha\phi\, \partial_\beta \phi
=-(\partial_k T^k )\Big(\big((\partial_0 +T^j \partial_j )\phi\big)^2
+\gamma^{ij}\partial_i \phi\,\partial_j \phi \Big)
\\
+2 \big((\partial_0 +T^k \partial_k )\phi\big)\Big(\big(\big(\partial_0 +T^k\partial_k\big)T^j
-\partial_i \gamma^{ij} \big)\partial_j\phi\Big)\\
+\big((\partial_0 +T^k \partial_k ) \gamma^{ij}\big)\partial_i\phi\, \partial_j\phi-2\gamma^{ij}(\partial_j T^k)\, \partial_i \phi\, \partial_k\phi
\end{multline}
and the inhomogeneous term is
\begin{equation}
G(t)=\int_0^t\int_{D_s} 2 \big((\partial_0 +T^k \partial_k )\phi\big) G \, dx ds \leq 2\int_0^t \sqrt{E(s)}\|G(s,\cdot)\|_{L^2(S_s)}\, ds
\end{equation}
It follows from the previous lemma that
\begin{multline}
H^{\alpha\beta}\xi_\alpha\xi_\beta=\big( (\xi_0+T^j\xi_j)^2+\gamma^{ij}\xi_i\xi_j \big)(1+T^k n_k)-2n_i \gamma^{ij}\xi_j
(\xi_0+T^j\xi_j)\\ \geq \big( (\xi_0+T^j\xi_j)^2+\gamma^{ij}\xi_i\xi_j \big)(1+T^k n_k)
-2\sqrt{\gamma^{ij}n_i n_j}\sqrt{ \gamma^{ij}\xi_i\xi_j}\,\big|\xi_0+T^j\xi_j\big|\\
\geq \big( (\xi_0+T^j\xi_j)^2+\gamma^{ij}\xi_i\xi_j \big)(1+T^k n_k)
-2(1-|T^k n_k|)\sqrt{ \gamma^{ij}\xi_i\xi_j}\,\big|\xi_0+T^j\xi_j\big|\geq  0
\end{multline}
 and hence
that the Flux is nonnegative. Moreover
$$
|R^{\alpha\beta}\partial_\alpha\phi\,\partial_\beta\phi|
\leq C_\varepsilon\big(|\partial T|+|\partial\gamma|\big) \Big( ((\partial_0+T^k\partial_k)\phi)^2+\gamma^{ij}\partial_i\phi\,\partial_j\phi\Big)
$$
We get the inequality
$$
E(t)\leq E(0)+\int_0^t C_\varepsilon n(s)\, E(s)+2 \sqrt{E(s)}\|G(s,\cdot)\|_{L^2(S_s)}\,  ds
$$
from which the lemma follows by a standard Gr{\"o}nwall argument.
\end{proof}

\begin{lemma}\label{lem:apriori} Let $g^{\alpha\beta}(X)$ be as in the previous lemma. Suppose that $|\hat{X}^0|\leq 1-\varepsilon$, for $0\leq t\leq T$ and  $X,h\in L^\infty([0,T],C^{2k}(\bold{R}^2))$ and $\phi\in L^\infty([0,T],H^k(\bold{R}^2))$.
Let
\begin{equation}
P\phi=\sum_{\alpha\beta=0,1,2} g^{\alpha\beta}(X)\partial_\alpha\partial_\beta \phi +\sum_{\beta=0,1,2} h^\beta\partial_\beta \phi.
\end{equation}
Then we have
\begin{equation}
\sum_{|\gamma|\leq 1} \|\partial^\gamma \phi(t,\cdot)\|_{H^s}\leq C\Big(\sum_{|\gamma|\leq 1}\|\partial^\gamma \phi(0,\cdot)\|_{H^s}+\int_0^t \|P\phi(\tau,\cdot)\|_{H^s}\, d\tau\Big) ,\qquad 0\leq t\leq T
\end{equation}
for any positive or negative integer $|s|\leq k$. Here $\|u\|_{H^s}^2=\int |\hat{u}(\xi)|^2(1+|\xi|^2)^{s} d\xi$, where $\hat{u}$ is the Fourier transform. The constant $C$ depends on $n(T), T, \varepsilon, X, h$. The inequality also applies for $s\leq k$ a nonnegative integer and\footnote{In this notation it is understood that also $X_t\in H^{k-1}$.} $X, h\in L^\infty([0,T],H^k(\bold{R}^2))$, provided $k\geq 4$.
\end{lemma}
\begin{proof} Following \cite{H} section 6.3 we differentiate the equation
with respect to $x$ derivatives only $\partial_x^\gamma=\partial_1^{\gamma_1} \partial_2^{\gamma_2}$
and use that the coefficient in front of $\partial_0^2$ is constant to obtain
\begin{equation}\label{eq1:lemma23}
P\partial_x^\gamma \phi
=\sum_{|\delta|\leq |\gamma|} f_\delta\cdot  \partial_x^\delta \partial\phi+\partial_x^\gamma P\phi
\end{equation}
and the result for positive $s$ follows from using the previous lemma in $\bold{R}^n\times [0,T]$,
together with that the $\|\phi(t,\cdot)\|_{L^2}$ is bounded by its value when $t=0$ plus a constant times
the time derivative in the interval.

For negative $s=-k$ we set
\begin{equation}
\psi=(I-\triangle_x)^{-k} \phi ,\qquad \phi=(I-\triangle_x)^k \psi
\end{equation}
Then
\begin{equation}\label{eq2:lemma23}
P\phi=
(I-\triangle_x)^k P \psi
-R\partial \psi
\end{equation}
where $R$ is a differential operator of the form
\begin{equation}
R\partial \psi =\sum_{|\gamma|,|\delta|\leq k}
\partial_x^\gamma \big( f_{\gamma\delta} \partial^\delta \cdot\partial\psi\big).
\end{equation}
We write this as
\begin{equation}
P\psi=(I-\triangle_x)^{-k} \big(R\partial\psi+ P\phi\big)
\end{equation}
By definition of the Sobolev norm for negative $s$
\begin{equation}
\|(I-\triangle_x)^{-k} P\phi\|_{H^{k}}\sim\|(I-\triangle_x)^{-k/2} P\phi\|_{L^2}=\|P\phi\|_{H^{-k}}
\end{equation}
The lemma for positive $s$ therefore gives
\begin{equation*}
\sum_{|\gamma|\leq 1}\|\partial^\gamma \psi(t,\cdot)\|_{H^k}\leq C\Big(\sum_{|\gamma|\leq 1} \|\partial^\gamma \psi(0,\cdot)\|_{H^k}
+\int_0^t \|R\partial\psi (\tau,\cdot)\|_{H^{-k}}+\|P\phi(\tau,\cdot)\|_{H^{-k}}\, d\tau\Big).
\end{equation*}
It follows from the particular form of $R$ that
\begin{equation}
\|R \partial\psi (t,\cdot)\|_{H^{-k}}\leq C\|\partial\psi(t,\cdot)\|_{H^k}
\end{equation}
and therefore by Gr{\"o}nwalls lemma that
\begin{equation}
\sum_{|\gamma|\leq 1}\|\partial^\gamma \psi(t,\cdot)\|_{H^k}\leq C^\prime\Big(\sum_{|\gamma|\leq 1}\|\partial^\gamma \psi(0,\cdot)\|_{H^k}
+\int_0^t \|P\phi (\tau,\cdot)\|_{H^{-k}}\, d\tau\Big).
\end{equation}
Since $\sum_{|\gamma|\leq 1}\|\partial^\gamma \phi(t,\cdot)\|_{H^{-k}}\sim\sum_{|\gamma|\leq 1}\|\partial^\gamma \psi(t,\cdot)\|_{H^k}$
this concludes the proof of the first part of the lemma.
\\
To conclude under the assumption $X\in L^\infty([0,T],H^k(\bold{R}^2)), k\geq 4$, and $s$ a nonnegative integer, observe that in \eqref{eq1:lemma23} we have 
\[
\sum_{|\delta|\leq |\gamma|}\|f_{\delta}\partial_x^{\delta}\partial\phi\|_{L_x^2}\lesssim \|\phi\|_{H^{k+1}}
\]
due to Sobolev's embedding $H^2(\R^2)\subset L^\infty(\R^2)$. 
\end{proof}

\subsection{Local existence}

\begin{lemma} Let $g^{\alpha\beta}(X)$ be as in the previous lemma. Suppose that $|\hat{X}^0|\leq 1-\varepsilon$, for $0\leq t\leq T$ and  $X\in L^\infty([0,T],C^{2k}(\bold{R}^2))$ and $g\in H^{k+1}(\bold{R}^2)$, $h\in H^k(\bold{R}^2)$
Then the equation
\begin{equation}
\sum_{\alpha\beta=0,1,2} g^{\alpha\beta}(X)\partial_\alpha\partial_\beta \phi=0,\qquad\phi\Big|_{t=0}=g,\quad
\partial_t\phi\Big|_{t=0}=h
\end{equation}
has a solution $\partial\phi \in L^\infty([0,T],H^k(\bold{R}^2))$.
\end{lemma}
\begin{proof} Following \cite{H} section 6.3, we introduce the adjoint operator
\begin{equation}
P^* \phi=\sum_{\alpha\beta=0,1,2} \partial_\alpha \partial_\beta\big( g^{\alpha\beta}(X)\phi\big).
\end{equation}
By the estimate in the previous lemma applied to $L^*$ with $t$ replaced by $T-t$ we have
\begin{equation}
\|\phi(t,\cdot)\|_{H^{-k}}\les \int_t^T \|P^* \phi (s,\cdot)\|_{H^{-k-1}}\, ds ,\qquad\phi\in C_0^\infty\big([-\infty,T)\times\bold{R}^2\big).
\end{equation}If $f\in L^1([0,T]; H^k(\bold{R}^2))$, then
\begin{equation}
|(f,\phi)|=\big|\int_0^T \big( f(t,\cdot),\phi(t,\cdot)\big)\, dt \big|\leq C\int_t^T \|P^* \phi (s,\cdot)\|_{H^{-k-1}}\, ds
\end{equation}
This therefore defines a linear functional $L(\psi)=(f,\phi)$ for all $\psi$ of the form
$\psi=P^* \phi$, for some $\phi\in C_0^\infty\big([-\infty,T)\times\bold{R}^2\big)$. Since this defines a linear
subspace of $L^1([-\infty,T]; H^{-k-1}(\bold{R}^2))$, where we have the bound \begin{equation}
|L(\psi)|\leq C\int_0^T \|\psi(s,\cdot)\|_{H^{-k-1}}\, ds.
\end{equation}
the functional can by the Hahn-Banach theorem be extended to the whole space without increasing the bound. Therefore there is an element in the dual space $u\in L^\infty([-\infty,T]; H^{k+1}(\bold{R}^2))$
such that 
\begin{equation}
L(\psi)=(u,\psi),\qquad \psi \in L^1([-\infty,T]; H^{-k-1}(\bold{R}^2)). 
\end{equation}
(That the dual space of $H^{-k}$ is $H^k$ follows from Parseval's formula and the fact that by Riesz Representation theorem its true for $L^2$.)
In view of the bound it follows that $u(t,x)=0$, for $t\leq 0$.  
In particular it follows that 
\begin{equation}
(f,\phi)=(u,P^*\phi)
\end{equation}
for all $\phi\!\in\! C_0^\infty([0,T)\!\times\bold{R}^2)$, i.e. $u$ is a distributional solution of the equation
$Pu\!=\!f$, when $0<t\leq T$, with vanishing Cauchy data. A solution with arbitrary Cauchy data is obtained if one choose any function $u_0$ with given data and introduce $u-u_0$ as unknown. \end{proof}

\begin{prop}\label{prop25} Suppose that $f$ and $h$ are smooth functions on $S_0=\{x; \, |x-x_0|\leq R+2\varepsilon\}$
such that
\begin{equation}
1+f_{x_1}^2+f_{x_2}^2 -h^2\geq 2\varepsilon(1+f_{x_1}^2+f_{x_2}^2),
\qquad\text{and}\qquad \|f\|_{H^5(S_0)}+\|h\|_{H^4(S_0)}\leq K.
\end{equation}
Then there is a $T_{\varepsilon, K}>0$ such that the initial value problem
$$
g^{\alpha\beta}(\partial\phi)\partial_{\alpha}\partial_\beta\phi =0
,\qquad \phi\big|_{t=0}=f,\quad \partial_t\phi\big|_{t=0}=h.
$$
has a solution in $D_{T_\varepsilon}=\cup_{0\leq t\leq T_\varepsilon} S_t$, where $S_t=\{(t,x);\, |x-x_0|\leq R-t\}$, satisfying
$1+\phi_{x_1}^2+\phi_{x_2}^2 -\phi_t^2\geq \varepsilon(1+\phi_{x_1}^2+\phi_{x_2}^2)$ and $\|\phi(t,\cdot)\|_{H^5(S_t)}+\|\phi_t(t,\cdot)\|_{H^4(S_t)}\leq 2K$, for $0\leq t\leq T_\varepsilon$. Any higher regularity of the initial data (i. e. the property to belong to $H^s$, $s>5$) is preserved. 
\end{prop}
\begin{proof} First we note that we can extend data outside the set $|x-x_0|<R$
so that the conditions on $(f,h)$ hold everywhere
and $(f,h)$ have compact support. Just multiply $h$ by a cutoff and then $f$ by another cutoff which is
unity on the support of the first cutoff. Outside a compact set our metric is then the Minkowski metric.
We now set up an iteration $\phi^0=0$ and for $k\geq 1$
\begin{equation}
g^{\alpha\beta}(\partial\phi^{k-1})\partial_\alpha\partial_\beta \phi^k=0
,\qquad \phi^k\big|_{t=0}=f,\quad \partial_t\phi^{k}\big|_{t=0}=h.
\end{equation}
The existence of (smooth) solutions to the linear equation above was given in the previous lemma.
What remains to show is that $\phi^k$ converges, which will follow from first proving that
the sequence is uniformly bounded with respect to $H^5$. This in turn will follow from the energy estimate, Lemma~\ref{lem:apriori} above
after first differentiating the equation to obtain equations and estimates for higher derivatives. This is a standard argument that can be found e.g. in \cite{H} section 6.3.
\end{proof}

\begin{theorem} Suppose that $S$ is a smooth surface ($2$ manifold in $3$ dimensional Euclidean space).
Suppose also that $\kappa$ is a smooth function on $S$ satisfying $|\kappa|<1$. Then there is $3$ manifold $M$
in $4$ dimensional Minkowski space satisfying the Minimal surface equation and such that $S=S_0$ , where $S_t=\{x; (t,x)\in M\}$ and $\kappa$ is the normal velocity of $S_t$, 
when $t=0$. 
\end{theorem}
\begin{proof} In local coordinates this becomes the problem above, and the solution has to be unique
in overlapping coordinate systems.
\end{proof}

For completeness' sake, we also shortly discuss a different coordinate system next:

\subsection{Cylindrical coordinates}
The gradient expressed in cylindrical coordinates in space $(t,r,z,\theta)$ is
$$
\nabla\Psi=-\Psi_t\partial_t+\Psi_r\partial_r
+r^{-1}\Psi_\theta\partial_\theta+\Psi_z\partial_z.
$$
The divergence in cylindrical coordinates is
$$
\nabla\cdot\bold{N}=\partial_t N_t+r^{-1}\partial_r (r N_r)+\partial_z
N_z+r^{-1}\partial_\theta N_\theta,
$$
which expressed in terms of $\Psi$ is
\begin{equation*}
\frac{\partial}{\partial t}\big[\frac{\Psi_t}{\sqrt{L}}\big] -
\frac{1}{r}\frac{\partial}{\partial r}\big[\frac{r\Psi_r}{\sqrt{L}}\big]
-\frac{1}{r^2}\frac{\partial}{\partial
\theta}\big[\frac{\Psi_\theta}{\sqrt{L}}\big]
-\frac{\partial}{\partial z}\big[\frac{\Psi_z}{\sqrt{L}}\big]=0,\qquad
L=|\nabla\Psi|
\end{equation*}

With $\Psi(t,x,y,z)=z-\phi(t,x,y)$ this reduces to
\eqref{eq:HypMin} expressed in cylindrical coordinates:
\begin{equation*}
\frac{\partial}{\partial t}\big[\frac{\phi_t}{\sqrt{L}}\big] -
\frac{1}{r}\frac{\partial}{\partial r}\big[\frac{r\phi_r}{\sqrt{L}}\big]
-\frac{1}{r^2}\frac{\partial}{\partial
\theta}\big[\frac{\phi_\theta}{\sqrt{L}}\big]=0,\qquad
L=1+\phi_r^2+\phi_\theta^2/r^2-\phi_t^2
\end{equation*}
However, we will rewrite our surface with $\Psi=r-\psi(t,z,\theta)$:
\begin{equation}\label{eq:cylindricalminimal}
\frac{\partial}{\partial t}\big[\frac{\psi_t}{\sqrt{L}}\big] -
\frac{1}{r}\frac{\partial}{\partial r}\big[\frac{r}{\sqrt{L}}\big]
-\frac{1}{r^2}\frac{\partial}{\partial
\theta}\big[\frac{\psi_\theta}{\sqrt{L}}\big]
-\frac{\partial}{\partial z}\big[\frac{\psi_z}{\sqrt{L}}\big]=0,\qquad
L=1+ \psi_z^2+\psi_\theta^2/r^2 - \psi_t^2
\end{equation}

\begin{lemma} The above equation takes they form
\begin{equation}
\widetilde{g}^{\,\alpha\beta}(\partial\psi)\partial_\alpha\partial_\beta \psi =\frac{1}{r}\big(1+ \psi_z^2+\frac{2}{r^2}\psi_\theta^2 - \psi_t^2\big),
\end{equation}
where the principal symbol is in the coordinates $(\partial_t,\partial_z,\partial_\theta/r)$ replaced by $(\tau,\zeta,\eta)$ is
\begin{equation*}
\widetilde{g}^{\,\alpha\beta}(\partial\psi)\xi_\alpha\xi_\beta
=(L+\psi_t^2)\tau^2 -(L-\psi_z^2)\zeta^2-\big(L-\tfrac{1}{r^2} \psi_\theta^2\big)\eta^2-2\psi_t\psi_t \tau\zeta
-2\psi_t\tfrac{1}{r}\psi_\theta \tau\eta+2\psi_z\tfrac{1}{r}\psi_\theta \zeta\eta .
\end{equation*}
We have
\begin{equation}
(L+\psi_t^2)^{-1} \widetilde{g}^{\,\alpha\beta}(\partial\psi)\xi_\alpha\xi_\beta
=(\tau-C\zeta-D\eta)^2 -\widetilde{\gamma}(\xi,\eta),
\end{equation}
where $\widetilde{\gamma}(\zeta,\eta)=c_0 \zeta^2+c_1\eta^2+c_2 \zeta\eta$ is positive definite if
\begin{equation}\label{eq:localexist1}
\psi>0,\qquad 1+\frac{1}{r^2}\psi_\theta^2+\psi_z^2-\psi_t^2>0.
\end{equation}
\end{lemma}
\begin{proof}
Simplifying gives
\begin{equation*}
 L\big(\psi_{tt}-\psi_{zz}-\frac{1}{r^2}\psi_{\theta\theta}-\frac{1}{r}\big)
-\frac{1}{2}\big(\psi_t L_t-L_r-\psi_z L_z -\frac{1}{r^2}\psi_\theta
L_\theta\big)=0,
\end{equation*}
or
\begin{multline*}
\big(1+ \psi_z^2+\psi_\theta^2/r^2 - \psi_t^2\big)
\big(\psi_{tt}-\psi_{zz}-\frac{1}{r^2}\psi_{\theta\theta}-\frac{1}{r}\big)
-\frac{1}{r^3}\psi_\theta^2=\\
\psi_t(\psi_z
\psi_{tz}+\frac{1}{r^2}\psi_{\theta}\psi_{t\theta}-\psi_t\psi_{tt}\big)
 -\psi_z
 \big(\psi_z\psi_{zz}+\frac{1}{r^2}\psi_\theta\psi_{z\theta}-\psi_t\psi_{zt}\big)
 -\frac{1}{r^2}\psi_\theta\big(\psi_z\psi_{\theta z}+\frac{1}{r^2}\psi_\theta
 \psi_{\theta\theta}-\psi_t\psi_{\theta t}\big).
\end{multline*}
 We have
\begin{multline*}
\big(1+ \psi_z^2+\frac{1}{r^2}\psi_\theta^2\big)\psi_{tt}
-\big(1+\frac{1}{r^2}\psi_\theta^2+\psi_z^2-\psi_t^2\big)\big(\psi_{zz}+\frac{1}{r^2}\psi_{\theta\theta}\big)\\
=2\psi_t\big(\psi_z
\psi_{tz}
+\frac{1}{r^2}\psi_{\theta}\psi_{t\theta}\big)
 -\big(\psi_z^2\psi_{zz}+\frac{1}{r^4}\psi_\theta^2\psi_{\theta\theta}+ \frac{2}{r^2}\psi_\theta \psi_z\psi_{\theta z}\big)
 +\frac{1}{r}\big(1+ \psi_z^2+\frac{2}{r^2}\psi_\theta^2 - \psi_t^2\big).
\end{multline*}
Replacing $(\partial_t,\partial_z,\partial_\theta/r)$ by $(\tau,\zeta,\eta)$
and dividing by $M=\big(1+ \psi_z^2+\frac{1}{r^2}\psi_\theta^2\big)$ we get the
characteristic polynomial:
$$
\tau^2-A\zeta^2-B\eta^2-2C \tau\zeta-2D\tau\eta+2E\zeta\eta,
$$
where
$$
A=\frac{1+\frac{1}{r^2}\psi_\theta^2-\psi_t^2}{M},\quad
B=\frac{1+\psi_z^2 -\psi_t^2}{M}\quad
C=\frac{\psi_t\psi_z}{M},\quad D=\frac{\psi_t\psi_{\theta}}{Mr},\quad E=\frac{\psi_\theta \psi_z}{Mr}
$$
Completing the squares we get
$$
(\tau-C\zeta-D\eta)^2-\big( (A+C^2)\zeta^2+(B+D^2)\eta^2-2(E-CD)\zeta\eta\big)
$$
This satisfies Garding's hyperbolicity condition if
the last polynomial is positive definite, i.e. if
$$
A+C^2>0,\quad B+D^2>0,\quad\text{and}\quad (A+C^2)(B+D^2)> (E-CD)^2
$$
We have
$$
M^2(A+C^2)=\big(1+ \psi_z^2+\frac{1}{r^2}\psi_\theta^2\big)
\big(1+\frac{1}{r^2}\psi_\theta^2-\psi_t^2\big)+\psi_t^2\psi_z^2
=(1+\frac{1}{r^2}\psi_\theta^2\big)
\big(1+\frac{1}{r^2}\psi_\theta^2+\psi_z^2-\psi_t^2)
$$
and
$$
M^2(B+D^2)=(1+\psi_z^2\big)
\big(1+\frac{1}{r^2}\psi_\theta^2+\psi_z^2-\psi_t^2)
$$
Moreover
$$
M^4( E-CD)^2=\big(1+ \psi_z^2+\frac{1}{r^2}\psi_\theta^2-\psi_t^2\big)^2\psi_z^2\frac{1}{r^2}\psi_\theta^2
$$
Hence
\begin{multline}
M^4 (A+C^2)(B+D^2)-M^4 (E-CD)^2=\\ \big(1+\psi_z^2+\frac{1}{r}\psi_\theta^2\big)
\big(1+\frac{1}{r^2}\psi_\theta^2+\psi_z^2-\psi_t^2)^2> 0,
\end{multline}
which proves the lemma.
\end{proof}

Hence in the case $\psi$ is independent of $\theta$
 \eqref{eq:cylindricalminimal} is hyperbolic as long as
\begin{equation}\label{eq:hyperbolic}
1+\psi_z^2-\psi_t^2>0,\qquad\text{and}\qquad \psi>0.
\end{equation}
In our case we are looking at a small perturbation $w$ of $\cosh{z}$:
$$
\psi=\cosh{z}+w
$$
If initial data $(w,w_t)|_{t=0}$ are small then \eqref{eq:hyperbolic} will hold,
 and local existence follows.

\section{Stability for perturbations away from the collar}
We study here radial perturbations of the static catenoid solution to the hyperbolic vanishing mean curvature flow which are supported far away from the 'collar' of the catenoid. We show that under a universal smallness assumption, a large class of such perturbations leads to solutions which exist 'until the perturbation reaches the collar'. Thus for these solutions any potential instability only sets in once the 'collar starts to move'. In the sequel all functions are of the form $f(t, r)$, $r = |x|$.

\begin{theorem} Let $Q(r) = \log[r+\sqrt{r^2-1}]$ the catenoid solution in polar coordinates, and consider perturbations at time $t=0$ of the form
\[
(\eps, \eps_t)|_{t=0} = \big(f(\frac{\cdot}{\lambda}),\,\lambda^{-1}g(\frac{\cdot}{\lambda})\big)
\]
for any $\lambda>1$, where $(f, g)\in C_0^\infty(1,2)$ and  we make the smallness assumption
\[
\sum_{1\leq\alpha\leq N}\|\partial_r^{\alpha}f(r)\|_{L^2_{rdr}} + \sum_{0\leq\alpha\leq N-1}\|\partial_r^{\alpha}g(r)\|_{L^2_{rdr}}<\kappa_0
\]
where $N\geq 10$ and $\kappa_0$ is sufficiently small, independent of $\lambda$. Then the solution $\phi(t, r) = Q(r)+\eps(t, r)$ with initial data
\[
(Q+\eps, \eps_t)|_{t=0}
\]
exists at least on the time interval $[0, \lambda - C_1]$ where $C_1$ is a universal constant (independent of the other parameters).
\end{theorem}

\begin{proof} We use Klainerman's method of commuting vector fields. Thus introduce the family of operators
\[
\Gamma_2 = t\partial_t + r\partial_r,\,\Gamma_1 = t\partial_{r} + r\partial_t,\,i=1,2
\]
We note that
\[
\sum_{1\leq \alpha\leq N-1}\|(\partial_r^\alpha\Gamma\eps)|_{t=0}\|_{L^2_{rdr}}\lesssim \kappa_0
\]
where $\Gamma$ stands for any one of the above vector fields. In light of Proposition~\ref{prop25}, the key will be the following
\begin{prop} Let $\delta>0$ small enough. Then provided $\kappa_0 = \kappa_0(\delta)>0$ is small enough, there exists a universal constant $K$ with the following property: assume that for any $T\in [0, \lambda-C_1]$, we have
\[
\sum_{1\leq|\alpha|\leq N}\|\la t\ra^{-\delta}\partial_{t,r}^{\alpha}\eps(t, \cdot)\|_{L_t^\infty L^2_{rdr}([0,T]\times \R_{+})}\leq K\kappa_0,
\]
\[
\sum_{1\leq|\alpha|\leq N-1}\sum_{\Gamma}\|\la t\ra^{-\delta}\partial_{t,r}^{\alpha}\Gamma\eps(t, \cdot)\|_{L_t^\infty L^2_{rdr}([0,T]\times \R_{+})}\leq K\kappa_0,
\]
\[
\sum_{1\leq|\alpha|\leq \frac{N}{2}+2}\|\langle t\rangle^{\frac{1}{2}}\partial_{t,r}^{\alpha}\eps(t, \cdot)\|_{L_{t,r}^\infty([0,T]\times \R_{+})}\leq K\kappa_0,
\]
Here $\partial_{t,r}^{\alpha}$ denotes all operators of the form $\partial_t^{\alpha_1}\partial_r^{\alpha_2}$, $\alpha_1 + \alpha_2 = \alpha$. Then one may replace $K$ by $\frac{K}{2}$ on the right hand side.
\end{prop}
\begin{proof}(Proposition) Write $\Box = \partial_t^2 - \partial_r^2 - \frac{1}{r}\partial_r$ and let $\phi(t, r) = Q(r)+\eps(t, r)$. We first derive the equation for $\eps$:
\begin{align*}
&\Box \eps - \triangle Q = \\
&\sqrt{1+|\nabla_x\phi|^2 - \eps_t^2}\big[-\eps_t\partial_t\big(\frac{1}{\sqrt{1+|\nabla_x\phi|^2 - \eps_t^2}}\big) + \sum_{i=1,2}\phi_{x_i}\partial_{x_i}\big(\frac{1}{\sqrt{1+|\nabla_x\phi|^2 - \eps_t^2}}\big)\big]
\end{align*}
We reformulate this as
\begin{align}
\Box \eps = &\sqrt{1+|\nabla_x\phi|^2 - \eps_t^2}\big[-\eps_t\partial_t\big(\frac{1}{\sqrt{1+|\nabla_x\phi|^2 - \eps_t^2}} - \frac{1}{\sqrt{1+|\nabla_xQ|^2}}\big)\nonumber\\
&\hspace{2.5cm}+\sum_{i=1,2}\eps_{x_i}\partial_{x_i}\big(\frac{1}{\sqrt{1+|\nabla_x\phi|^2 - \eps_t^2}} - \frac{1}{\sqrt{1+|\nabla_xQ|^2}}\big)\label{eq:epseqnullform}\\
&\hspace{2.5cm}+\sum_{i=1,2}Q_{x_i}\partial_{x_i}\big(\frac{1}{\sqrt{1+|\nabla_x\phi|^2 - \eps_t^2}} - \frac{1}{\sqrt{1+|\nabla_x Q|^2}}\big)\label{eq:epseqerror1}\\
&\hspace{2.5cm}+\sum_{i=1,2}\eps_{x_i}\partial_{x_i}\big(\frac{1}{\sqrt{1+Q_r^2}}\big)\label{eq:epseqerror11}\big]\\
&+(\sqrt{1+|\nabla_x\phi|^2 - \eps_t^2} - \sqrt{1+|\nabla_x Q|^2})\sum_{i=1,2}Q_{x_i}\partial_{x_i}[\frac{1}{\sqrt{1+|\nabla_x Q|^2}}]\label{eq:epseqerror2}
\end{align}
where we have exploited the fact that $Q$ is a static solution, i. e.
\[
\frac{\triangle Q}{\sqrt{1+|\nabla_x Q|^2}} = -\sum_{i=1,2}Q_{x_i}\partial_{x_i}\big(\frac{1}{\sqrt{1+|\nabla_x Q|^2}}\big)
\]
As all functions are radial, we compute
\begin{align*}
&\sum_{i=1,2}Q_{x_i}\partial_{x_i}\big(\frac{1}{\sqrt{1+|\nabla_x\phi|^2 - \eps_t^2}} - \frac{1}{\sqrt{1+|\nabla_x Q|^2}}\big)\\& = -\frac{Q_r}{2}\big[\frac{\partial_r[(Q_r)^2 + 2Q_r\eps_r + \eps_r^2 - \eps_t^2}{(1+|\nabla_x\phi|^2 - \eps_t^2)^{\frac{3}{2}}} - \frac{\partial_r(Q_r^2)}{(1+|\nabla_x Q|^2)^{\frac{3}{2}}}\big]\\
& = O(\frac{1}{r^4})\big[\frac{1}{(1+|\nabla_x\phi|^2 - \eps_t^2)^{\frac{3}{2}}} - \frac{1}{(1+Q_r^2)^{\frac{3}{2}}}\big]\\
&+O(\frac{1}{r^2})\frac{\eps_{rr} + O(\frac{1}{r})\eps_r}{(1+|\nabla_x\phi|^2 - \eps_t^2)^{\frac{3}{2}}} + O(\frac{1}{r})\frac{\eps_r\eps_{rr} - \eps_t\eps_{tr}}{(1+|\nabla_x\phi|^2 - \eps_t^2)^{\frac{3}{2}}}
\end{align*}
Here the expressions $O(\frac{1}{r^k})$ depend only on $r$ and have symbol type behavior. \\
Further, since $\partial_r\big(\frac{1}{\sqrt{1+Q_r^2}}\big) =O(\frac{1}{r^3})$, we have
\[
\sum_{i=1,2}\eps_{x_i}\partial_{x_i}\big(\frac{1}{\sqrt{1+Q_r^2}}\big) = O(\frac{1}{r^3})\eps_r
\]
We then formulate the equation for $\eps$ as
\begin{align}
\Box\eps =&\sqrt{1+|\nabla_x\phi|^2 - \eps_t^2}\big[-\eps_t\partial_t\big(\frac{1}{\sqrt{1+|\nabla_x\phi|^2 - \eps_t^2}} - \frac{1}{\sqrt{1+|\nabla_xQ|^2}}\big)\nonumber\\
&\hspace{2.5cm}+\sum_{i=1,2}\eps_{x_i}\partial_{x_i}\big(\frac{1}{\sqrt{1+|\nabla_x\phi|^2 - \eps_t^2}} - \frac{1}{\sqrt{1+|\nabla_xQ|^2}}\big)\label{eq:epseqnnullform}\\
&\hspace{2.5cm}+ O(\frac{1}{r^2})\frac{\eps_{rr} + O(\frac{1}{r})\eps_r}{(1+|\nabla_x\phi|^2 - \eps_t^2)^{\frac{3}{2}}} + O(\frac{1}{r})\frac{\eps_r\eps_{rr} - \eps_t\eps_{tr}}{(1+|\nabla_x\phi|^2 - \eps_t^2)^{\frac{3}{2}}}\label{eq:epseqnerrors1}\\
&\hspace{2.5cm}+O(\frac{1}{r^3})\eps_r +O(\frac{1}{r^4})\big[\frac{1}{(1+|\nabla_x\phi|^2 - \eps_t^2)^{\frac{3}{2}}} - \frac{1}{(1+Q_r^2)^{\frac{3}{2}}}\big]\big]\label{eq:epseqnerrors2}\\
&+(\sqrt{1+|\nabla_x\phi|^2 - \eps_t^2} - \sqrt{1+|\nabla_x Q|^2})\sum_{i=1,2}Q_{x_i}\partial_{x_i}[\frac{1}{\sqrt{1+|\nabla_x Q|^2}}]\label{eq:epseqnerrors3}
\end{align}
{\bf{Energy estimates:}} We start by bootstrapping the bound
\begin{equation}\label{eq:en}
\sum_{1\leq\alpha\leq N}\|\partial_{t,r}^{\alpha}\eps(t, \cdot)\|_{L_t^\infty L^2_{rdr}([0,T]\times \R_{+})}\leq K\kappa_0\la t\ra^{\delta}
\end{equation}
Thus let $D^{\alpha} = \partial_t^{\beta_1}\partial_{r}^{\beta_2}$, with $1\leq \sum\beta_i= |\alpha|\leq N-1$.
Write schematically (we suppress constant coefficients)
\begin{align}
&\Box (D^{\alpha}\eps) + [D^{\alpha}, \Box]\eps\label{eq:bigmess}\\& = \sum_{\sum\alpha_i = \alpha}D^{\alpha_1}(\sqrt{1+|\nabla_x\phi|^2 - \eps_t^2})\big[-D^{\alpha_2}(\eps_t)\partial_t D^{\alpha_3}\big(\frac{1}{\sqrt{1+|\nabla_x\phi|^2 - \eps_t^2}} - \frac{1}{\sqrt{1+|\nabla_xQ|^2}}\big)\nonumber\\
&+\sum_{i=1,2}D^{\alpha_2}(\eps_{x_i})D^{\alpha_3}\partial_{x_i}\big(\frac{1}{\sqrt{1+|\nabla_x\phi|^2 - \eps_t^2}} - \frac{1}{\sqrt{1+|\nabla_xQ|^2}}\big)\label{eq:ennullform}\\
&+ D^{\alpha_2}\big(O(\frac{1}{r^2})\big) D^{\alpha_3}\big(\frac{\eps_{rr} + O(\frac{1}{r})\eps_r}{(1+|\nabla_x\phi|^2 - \eps_t^2)^{\frac{3}{2}}}\big) + D^{\alpha_2}\big(O(\frac{1}{r})\big)D^{\alpha_3}\big(\frac{\eps_r\eps_{rr} - \eps_t\eps_{tr}}{(1+|\nabla_x\phi|^2 - \eps_t^2)^{\frac{3}{2}}}\big)\label{eq:enerror1}\\
&+D^{\alpha_2}\big(O(\frac{1}{r^3})\big)D^{\alpha_3}\eps_r +D^{\alpha_2}\big(O(\frac{1}{r^4})\big)D^{\alpha_3}\big[\frac{1}{(1+|\nabla_{x}\phi|^2 - \eps_t^2)^{\frac{3}{2}}} - \frac{1}{(1+Q_r^2)^{\frac{3}{2}}}\big]\big]\label{eq:enerror2}\\
&+D^{\alpha_1}(\sqrt{1+|\nabla_x\phi|^2 - \eps_t^2} - \sqrt{1+|\nabla_x Q|^2})\sum_{i=1,2}D^{\alpha_2}Q_{x_i}D^{\alpha_3}\partial_{x_i}[\frac{1}{\sqrt{1+|\nabla_x Q|^2}}]\label{eq:enerror3}
\end{align}
In order to improve the bound \eqref{eq:en}, we use the standard energy method for free waves. Specifically, writing
\[
X_{\alpha}(t): = \frac{1}{2}\int_0^\infty (|D^{\alpha}\eps_t(t, \cdot)|^2 + |\partial_rD^{\alpha}\eps(t, \cdot)|^2)\,rdr,
\]
we get
\[
X'_{\alpha}(t) = \sum_{i=1}^4A^i_{\alpha}(t) - \int_0^\infty [D^{\alpha}, \Box]\eps D^{\alpha}\eps_t\,rdr
\]
where the $A^i_\alpha$ are the contributions corresponding to the terms \eqref{eq:ennullform} to \eqref{eq:enerror3}, i. e. writing the latter as $X_\alpha^1,\ldots, X_\alpha^4$, we have
\[
A^i_{\alpha} = \int_0^\infty X_{\alpha}^i D^{\alpha}\eps_t\,rdr
\]

{\it{Contribution of $A^1_\alpha$.}} We treat the first line of \eqref{eq:ennullform}, the second being similar. This contribution is the integral
\begin{align}
&\sum_{\sum\alpha_i = \alpha}\int_0^t\int_0^\infty D^{\alpha_1}(\sqrt{1+|\nabla_x\phi|^2 - \eps_t^2})D^{\alpha_2}(\eps_t)\nonumber\\&\hspace{4cm}\times\partial_t D^{\alpha_3}\big(\frac{1}{\sqrt{1+|\nabla_x\phi|^2 - \eps_t^2}} - \frac{1}{\sqrt{1+|\nabla_xQ|^2}}\big)D^{\alpha}\eps_t\,rdrdt\label{eq:A_alpha1}
\end{align}
Using the Huyghen's principle which implies $r\geq \lambda - t$ on the support of $\eps$, we obtain the bound
\begin{equation}\label{eq:unboosted1}
\big|\partial_t D^{\alpha_3}\big(\frac{1}{\sqrt{1+|\nabla_x\phi|^2 - \eps_t^2}} - \frac{1}{\sqrt{1+|\nabla_xQ|^2}}\big)\big|\lesssim [(\lambda - t)^{-1}+K\kappa_0\la t\ra^{-\frac{1}{2}}]\sum_{|\alpha|\leq |\alpha_3|+2}|D^{\alpha}\eps|
\end{equation}
Then we divide the above integral into the cases $\alpha_3 = \alpha$, $|\alpha|-1\geq|\alpha_3|\geq \frac{N}{2}$, $|\alpha_3|<\frac{N}{2}$. We first deal with the latter two. In the second situation, we have $|\alpha_1|+|\alpha_2|\leq \frac{N}{2}$, and so (pointwise bound)
\[
\big|D^{\alpha_1}(\sqrt{1+|\nabla_x\phi|^2 - \eps_t^2})D^{\alpha_2}(\eps_t)\big|\lesssim \la t\ra^{-\frac{1}{2}}K\kappa_0
\]
In this case, using also \eqref{eq:unboosted1}, the above integral \eqref{eq:A_alpha1} is then bounded by
\[
\lesssim \int_0^t (K\kappa_0)^3\la s\ra^{2\delta}\la s\ra^{-\frac{1}{2}} [(\lambda - s)^{-1}+K\kappa_0\la s\ra^{-\frac{1}{2}}]\,ds\lesssim [\delta^{-1}(K\kappa_0)^4 + (K\kappa_0)^3\frac{\log\lambda}{\lambda^{\frac{1}{2}}}] \la t\ra^{2\delta}
\]
whence we can close provided $\delta^{-1}(K\kappa_0)^4 + (K\kappa_0)^3\frac{\log\lambda}{\lambda^{\frac{1}{2}}}\ll (K\kappa_0)^2$.
\\
In the third situation when $|\alpha_3|<\frac{N}{2}$, we have
\begin{equation}\label{eq:-1}
\|\partial_t D^{\alpha_3}\big(\frac{1}{\sqrt{1+|\nabla_x\phi|^2 - \eps_t^2}} - \frac{1}{\sqrt{1+|\nabla_xQ|^2}}\big)\|_{L_{r}^\infty}\lesssim K\kappa_0[(\lambda - t)^{-1}+K\kappa_0\la t\ra^{-\frac{1}{2}}]\la t\ra^{-\frac{1}{2}},
\end{equation}
while also
\begin{equation}\label{eq:0}
\|D^{\alpha_1}(\sqrt{1+|\nabla_x\phi|^2 - \eps_t^2})D^{\alpha_2}(\eps_t)\|_{L^2_{rdr}}\lesssim K\kappa_0\la t\ra^{\delta},
\end{equation}
and so the corresponding contribution to \eqref{eq:A_alpha1} is also bounded by 
\[
\lesssim \int_0^t (K\kappa_0)^3\la s\ra^{2\delta}\la s\ra^{-\frac{1}{2}} [(\lambda - s)^{-1}+K\kappa_0\la s\ra^{-\frac{1}{2}}]\,ds\lesssim [\delta^{-1}(K\kappa_0)^4 + (K\kappa_0)^3\frac{\log\lambda}{\lambda^{\frac{1}{2}}}] \la t\ra^{2\delta}
\]
Finally, consider the case $\alpha_3 = \alpha$. We write schematically
\begin{align}\label{eq:1}
\partial_t D^{\alpha}\big(\frac{1}{\sqrt{1+|\nabla_x\phi|^2 - \eps_t^2}} - \frac{1}{\sqrt{1+|\nabla_xQ|^2}}\big) =F_1 D^{\alpha}\eps_{tt} + F_2 D^{\alpha}\eps_{tr}+F_3,
\end{align}
where we have
\[
\|\partial_t^\beta F_1\|_{L_r^\infty} + \|\partial_r^\beta F_2\|_{L_r^\infty}  \lesssim (\lambda - t)^{-1} + K\kappa_0\la t\ra^{-\frac{1}{2}},\,\beta=0,1
\]
\[
\|F_3\|_{L_{rdr}^2}\lesssim K\kappa_0\la t\ra^{\delta}[(\lambda - t)^{-1} + K\kappa_0\la t\ra^{-\frac{1}{2}}]
\]
In order to deal with the contributions of the top derivatives $ D^{\alpha}\eps_{tt}, D^{\alpha}\eps_{tr}$, we use integration by parts:

\begin{align*}
&\int_0^t\int_0^\infty (\sqrt{1+|\nabla_x\phi|^2 - \eps_t^2})\eps_t F_1 D^{\alpha}\eps_{tt}D^{\alpha}\eps_t\,rdrdt\\
&= \frac{1}{2}\int_0^\infty (\sqrt{1+|\nabla_x\phi|^2 - \eps_t^2})\eps_t F_1 \big(D^{\alpha}\eps_t\big)^2\,rdr|_{0}^t\\
&-\frac{1}{2}\int_0^t\int_0^\infty \partial_t\big[(\sqrt{1+|\nabla_x\phi|^2 - \eps_t^2})\eps_t F_1\big] \big(D^{\alpha}\eps_t\big)^2\,rdrdt\\
\end{align*}
and we have
\[
\int_0^\infty (\sqrt{1+|\nabla_x\phi|^2 - \eps_t^2})\eps_t F_1 \big(D^{\alpha}\eps_t\big)^2\,rdr|_{0}^t\lesssim K\kappa_0\la t\ra^{-\frac{1}{2}}\|D^{\alpha}\eps_t(t, \cdot)\|_{L^2_{rdr}}^2 + \kappa_0^3\ll (K\kappa_0)^2\la t\ra^{2\delta}
\]
and further
\begin{align*}
&\big|\int_0^t\int_0^\infty \partial_t\big[(\sqrt{1+|\nabla_x\phi|^2 - \eps_t^2})\eps_t F_1\big] \big(D^{\alpha}\eps_t\big)^2\,rdrdt\big|\\
&\lesssim \int_0^t (K\kappa_0)^3\la s\ra^{2\delta}[(\lambda - s)^{-1} + K\kappa_0\la s\ra^{-\frac{1}{2}}]\la s\ra^{-\frac{1}{2}}\,ds\\
&\lesssim \delta^{-1}(K\kappa_0)^4\la t\ra^{2\delta} + (K\kappa_0)^3\frac{\log\lambda}{\lambda^{\frac{1}{2}}}\ll (K\kappa_0)^2\la t\ra^{2\delta}
\end{align*}
if we choose $\kappa_0$ small enough (in relation to $\delta$ as well as in an absolute sense). \\

For the contribution of the second top derivative term, we find
\begin{align*}
&\int_0^t\int_0^\infty (\sqrt{1+|\nabla_x\phi|^2 - \eps_t^2})\eps_t F_2 \partial_r\big(D^{\alpha}\eps_{t}\big)D^{\alpha}\eps_t\,rdrdt\\
&=-\int_0^t\int_0^\infty \frac{1}{2r}\partial_r[r(\sqrt{1+|\nabla_x\phi|^2 - \eps_t^2})\eps_t F_2]\big(D^{\alpha}\eps_{t}\big)D^{\alpha}\eps_t\,rdrdt\\
\end{align*}
and we have
\[
\big|\frac{1}{2r}\partial_r[r(\sqrt{1+|\nabla_x\phi|^2 - \eps_t^2})\eps_t F_2](t, \cdot)\big|\lesssim K\kappa_0\la t\ra^{-\frac{1}{2}}[(\lambda - t)^{-1} + K\kappa_0\la t\ra^{-\frac{1}{2}}]
\]
It follows that
\begin{align*}
&\big|\int_0^t\int_0^\infty (\sqrt{1+|\nabla_x\phi|^2 - \eps_t^2})\eps_t [F_3 + F_2\big(D^{\alpha}\eps_{tr}\big)]D^{\alpha}\eps_t\,rdrdt\big|\\
&\lesssim \int_0^t K\kappa_0\la s\ra^{-\frac{1}{2}}(K\kappa_0)^2\la s\ra^{2\delta}[(\lambda - s)^{-1} + K\kappa_0\la s\ra^{-\frac{1}{2}}]\,ds\ll (K\kappa_0)^2\la t\ra^{2\delta},\,0\leq t\leq \lambda - C_1,
\end{align*}
provided $\kappa_0$ is small enough in relation to $\delta$, as well as in absolute size.
This concludes treating the contribution of $A^1_\alpha$.
\\

{\it{Contribution of $A^2_\alpha$}}. We again distinguish between the cases $|\alpha_3|<|\alpha|$, $\alpha_3 = \alpha$. In the former case, we get provided $|\alpha|\geq |\alpha_3|>\frac{N}{2}$ the inequalities
\[
 \|D^{\alpha_2}\big(O(\frac{1}{r^2})\big) D^{\alpha_3}\big(\frac{\eps_{rr} + O(\frac{1}{r})\eps_r}{(1+|\nabla_x\phi|^2 - \eps_t^2)^{\frac{3}{2}}}\big)\|_{L^2_{rdr}}\lesssim (\lambda - t)^{-2}K\kappa_0 \la t\ra^{\delta},
\]
\[
\|D^{\alpha_2}\big(O(\frac{1}{r})\big)D^{\alpha_3}\big(\frac{\eps_r\eps_{rr} - \eps_t\eps_{tr}}{(1+|\nabla_x\phi|^2 - \eps_t^2)^{\frac{3}{2}}}\big)\|_{L^2_{rdr}}\lesssim (K\kappa_0)^2(\lambda - t)^{-1}\la t\ra^{-\frac{1}{2}}\la t\ra^{\delta}
\]
\[
\|D^{\alpha_1}(\sqrt{1+|\nabla_x\phi|^2 - \eps_t^2})\|_{L^\infty_{rdr}}\lesssim 1
\]
and if $\alpha_3\leq \frac{N}{2}$, we have
\[
\|D^{\alpha_2}\big(O(\frac{1}{r^2})\big) D^{\alpha_3}\big(\frac{\eps_{rr} + O(\frac{1}{r})\eps_r}{(1+|\nabla_x\phi|^2 - \eps_t^2)^{\frac{3}{2}}}\big)\|_{L^\infty_{rdr}}\lesssim (K\kappa_0)(\lambda - t)^{-2}\la t\ra^{-\frac{1}{2}}
\]
\[
\|D^{\alpha_2}\big(O(\frac{1}{r})\big)D^{\alpha_3}\big(\frac{\eps_r\eps_{rr} - \eps_t\eps_{tr}}{(1+|\nabla_x\phi|^2 - \eps_t^2)^{\frac{3}{2}}}\big)\|_{L^\infty_{rdr}}\lesssim (K\kappa_0)^2(\lambda - t)^{-1}\la t\ra^{-1}
\]
\[
\|D^{\alpha_1}(\sqrt{1+|\nabla_x\phi|^2 - \eps_t^2})\|_{L^2_{rdr}}\lesssim (K\kappa_0)\langle t\rangle^{\delta}[(\lambda - t)^{-1}+(K\kappa_0)\langle t\rangle^{-\frac{1}{2}}]
\]
We infer that
\begin{align*}
|\int_0^t A_\alpha^2(s)\,ds|&\leq \int_0^t\int_0^\infty |X_\alpha^2| |D^\alpha\eps_s|\,r\,drds\\
&\lesssim \int_0^t (K\kappa_0)^2\la s\ra^{2\delta}(\lambda- s)^{-1}\big[(K\kappa_0)\la s\ra^{-\frac{1}{2}}+ (\lambda - s)^{-1}\big]\,ds\\
&\ll (K\kappa_0)^2\la t\ra^{2\delta},\,0\leq t\leq \lambda - C_1
\end{align*}
provided we choose $C_1$ sufficiently large.
On the other hand, if $\alpha_3 = \alpha$, only the case when all derivatives fall on $\eps_{rr}$ needs to be considered, as the remaining cases are treated in the situation $|\alpha_3|<|\alpha|$. Thus this is the contribution of the terms
\begin{equation}\label{eq:2}
O(\frac{1}{r^2})\frac{D^{\alpha}\eps_{rr}}{(1+|\nabla_x\phi|^2 - \eps_t^2)^{\frac{3}{2}}},\, O(\frac{1}{r})\frac{\eps_r D^{\alpha}\eps_{rr}}{(1+|\nabla_x\phi|^2 - \eps_t^2)^{\frac{3}{2}}}
\end{equation}
Here we perform integration by parts twice:
\begin{align}
&\int_0^t\int_0^\infty(\sqrt{1+|\nabla_x\phi|^2 - \eps_t^2})O(\frac{1}{r^2})\frac{\partial_r D^{\alpha}\eps_{r}}{(1+|\nabla_x\phi|^2 - \eps_t^2)^{\frac{3}{2}}}D^{\alpha}\eps_t\,rdrdt\label{eq:(1)}\\
& = -\int_0^t\int_0^\infty(\sqrt{1+|\nabla_x\phi|^2 - \eps_t^2})O(\frac{1}{r^2})\frac{D^{\alpha}\eps_{r}}{(1+|\nabla_x\phi|^2 - \eps_t^2)^{\frac{3}{2}}}\partial_r D^{\alpha}\eps_t\,rdrdt\nonumber\\
&-\int_0^t\int_0^\infty\frac{1}{r}\partial_r\big[\frac{1}{(1+|\nabla_x\phi|^2 - \eps_t^2)}O(\frac{1}{r})\big]D^{\alpha}\eps_{r}D^{\alpha}\eps_t\,rdrdt\nonumber
\end{align}
\begin{align}
&\int_0^t\int_0^\infty(\sqrt{1+|\nabla_x\phi|^2 - \eps_t^2})O(\frac{1}{r})\frac{\eps_r D^{\alpha}\eps_{rr}}{(1+|\nabla_x\phi|^2 - \eps_t^2)^{\frac{3}{2}}}D^{\alpha}\eps_t\,rdrdt\label{eq:(2)}\\
&=-\int_0^t\int_0^\infty(\sqrt{1+|\nabla_x\phi|^2 - \eps_t^2})O(\frac{1}{r})\frac{\eps_r D^{\alpha}\eps_{r}}{(1+|\nabla_x\phi|^2 - \eps_t^2)^{\frac{3}{2}}}\partial_r D^{\alpha}\eps_t\,rdrdt\nonumber\\
&-\int_0^t\int_0^\infty \frac{1}{r}\partial_r\big[O(1)\frac{\eps_r }{(1+|\nabla_x\phi|^2 - \eps_t^2)}\big]D^{\alpha}\eps_{r}D^{\alpha}\eps_t\,rdrdt\nonumber
\end{align}
For the first terms after the equality sign in these two equations, we perform an integration by parts with respect to $t$, thereby obtaining the expressions
\begin{align*}
\eqref{eq:(1)} =& -\frac{1}{2}\int_0^\infty O(\frac{1}{r^2})\frac{1}{(1+|\nabla_x\phi|^2 - \eps_t^2)}|D^{\alpha}\eps_r|^2\,rdr|_{0}^t\\
&+\frac{1}{2}\int_0^t\int_0^\infty\partial_t\big(\frac{1}{(1+|\nabla_x\phi|^2 - \eps_t^2)}\big)O(\frac{1}{r^2})|D^{\alpha}\eps_r|^2\,rdr\,dt\\
&-\int_0^t\int_0^\infty\frac{1}{r}\partial_r\big[\frac{1}{(1+|\nabla_x\phi|^2 - \eps_t^2)}O(\frac{1}{r})\big]D^{\alpha}\eps_{r}D^{\alpha}\eps_t\,rdrdt\\
\end{align*}
as well as
\begin{align*}
\eqref{eq:(2)} =& -\frac{1}{2}\int_0^\infty O(\frac{1}{r})\frac{\eps_r}{(1+|\nabla_x\phi|^2 - \eps_t^2)}|D^{\alpha}\eps_r|^2\,rdr|_{0}^t\\
&+\frac{1}{2}\int_0^t\int_0^\infty O(\frac{1}{r})\partial_t\big(\frac{\eps_r}{(1+|\nabla_x\phi|^2 - \eps_t^2)}\big)|D^{\alpha}\eps_r|^2\,rdr dt\\
&-\int_0^t\int_0^\infty \frac{1}{r}\partial_r\big[O(1)\frac{\eps_r }{(1+|\nabla_x\phi|^2 - \eps_t^2)}\big]D^{\alpha}\eps_{r}D^{\alpha}\eps_t\,rdrdt\\
\end{align*}
The first combination of terms is bounded by
\begin{align*}
|\eqref{eq:(1)}|\lesssim &(\lambda - t)^{-2}(K\kappa_0)^2 \la t\ra^{2\delta} + \int_0^t (K\kappa_0)^3 \la s\ra^{2\delta-\frac{1}{2}}(\lambda - s)^{-3}\,ds\\
&\ll (K\kappa_0)^2\la t\ra^{2\delta},\,0\leq t\leq \lambda - C_1,
\end{align*}
while the second combination of terms is bounded by
\begin{align*}
|\eqref{eq:(2)}|\lesssim & (\lambda - t)^{-1}\la t\ra^{-\frac{1}{2}}(K\kappa_0)^3 \la t\ra^{2\delta} + \int_0^t (K\kappa_0)^3 \la s\ra^{2\delta-\frac{1}{2}}(\lambda - s)^{-1}\,ds\\
&\ll (K\kappa_0)^2\la t\ra^{2\delta},\,0\leq t\leq \lambda - C_1
\end{align*}
as desired.
\\

{\it{Contribution of $A_{\alpha}^{3,4}$.}} These can easily be handled as in the situation $|\alpha_3|<|\alpha|$ for the preceding term $A_\alpha^{2}$.
\\

Finally, to complete the (plain) energy bootstrap, we still need to bound the contribution of the commutator term $[D^{\alpha}, \Box]\eps$. It is immediate to verify that
\[
[D^{\alpha}, \Box]\eps = \sum_{|\beta|\leq |\alpha|-1} F_\beta (r)D^{\beta}\eps_r
\]
with $|F_{\beta}|\lesssim r^{-1-|\alpha-\beta|}\lesssim r^{-2}$. Then we get the bound
\begin{align*}
|\int_0^t \int_0^\infty [D^\alpha, \Box]\eps D^{\alpha}\eps_t\,rdr\,dt|&\lesssim \int_0^t (\lambda - s)^{-2}(K\kappa_0)^2\la s\ra^{2\delta}\,ds\\
&\ll (K\kappa_0)^2\la t\ra^{2\delta},\,0\leq t\leq \lambda - C_1
\end{align*}
provided $C_1$ is chosen large enough. This concludes the bootstrap for the bound \eqref{eq:en}.
\\

{\bf{Boosted energy bounds}}. Here we improve the estimate
\begin{align}\label{eq:boosten}
\sum_{1\leq\alpha\leq N-1}\sum_{\Gamma}\|\la t\ra^{-\delta}\partial_{t,r}^{\alpha}\Gamma\eps(t, \cdot)\|_{L_t^\infty L^2_{rdr}([0,T]\times \R_{+})}\leq K\kappa_0
\end{align}
We mimic the process used for the energy bounds, but this time with $D^\alpha = \partial_t^{\beta_1}\partial_r^{\beta_2}\Gamma_{1,2}$. To begin with, note that we no longer have the simple bound \eqref{eq:unboosted1}, but instead the more complicated
\begin{align}
&\big|\partial_{t,r} D^{\alpha_3}\big(\frac{1}{\sqrt{1+|\nabla_x\phi|^2 - \eps_t^2}} - \frac{1}{\sqrt{1+|\nabla_xQ|^2}}\big)\big|\label{eq:boosted1}\\&\lesssim [(\lambda - t)^{-1}+K\kappa_0\la t\ra^{-\frac{1}{2}}]\sum_{1\leq|\alpha|\leq |\alpha_3|+2}r^{-s(\alpha)}|D^{\alpha}\eps|,\,|\alpha_3|\leq \frac{N}{2}\nonumber
\end{align}
\begin{align}
&\big|\partial_{t,r} D^{\alpha_3}\big(\frac{1}{\sqrt{1+|\nabla_x\phi|^2 - \eps_t^2}} - \frac{1}{\sqrt{1+|\nabla_xQ|^2}}\big)\big|\label{eq:boosted2}\\&\lesssim [(\lambda - t)^{-1}+K\kappa_0\langle t\rangle^{\delta}(\lambda - t)^{-\frac{1}{2}}]\sum_{1\leq|\alpha|\leq |\alpha_3|+2}r^{-s(\alpha)}|D^{\alpha}\eps|,\,N-2\geq|\alpha_3|>\frac{N}{2}\nonumber
\end{align}

where now $D^\alpha_3$ etc may involve one operator of the form $\Gamma_{1,2}$, and we have
\[
s(\alpha) = 1,\,\text{provided}\,D^{\alpha} = \Gamma_{1,2},
\]
and $s(\alpha) = 0$ otherwise. To see this, note that if the operator $\Gamma_{1,2}$ falls on the first factor in a product term $\nabla_{x}Q\cdot\nabla_x\eps = Q_r\eps_r$, then we have $\Gamma_2 Q_r\eps_r = O(\frac{1}{r})\eps_r$, while
$\Gamma_1Q_r\eps_r = O(\frac{1}{r^2})(\Gamma_1\eps - r\eps_t) = O(\frac{1}{r})\frac{\Gamma_1\eps}{r} - O(\frac{1}{r})\eps_t$. The reason for the term factor $K\kappa_0\langle t\rangle^{\delta}(\lambda - t)^{-\frac{1}{2}}$ is the fact that the operator $\Gamma_{1,2}$ may fall on one factor $\eps_{t,r}$ while the remaining $\partial_t^{\beta_1}\partial_r^{\beta_2}$ may fall on another factor $\eps_{t,r}$, and a priori we only have an $L^2$-bound at our disposal for this in case $|\alpha_3|>\frac{N}{2}$. However, since
\[
\beta_1+\beta_2 = |\alpha_3|-1\leq N-3,
\]
we can the use the radial Sobolev embedding and our support assumptions to get
\[
|\partial_{t,r}\partial_t^{\beta_1}\partial_r^{\beta_2}\eps_{t,r}(t, r)|\lesssim r^{-\frac{1}{2}}\|\partial_{t,r}\partial_t^{\beta_1}\partial_r^{\beta_2}\eps_{t,r}(t, \cdot)\|_{H^1_{rdr}}\lesssim (K\kappa_0)(\lambda - t)^{-\frac{1}{2}}\langle t\rangle^{\delta}
\]

Now we estimate the same four contributions as for the energy bounds:
\\

Using \eqref{eq:ennullform} - \eqref{eq:enerror3}, we commence with
\\

{\it{Contribution of $A_{\alpha}^1$}}; here we use the same notation as before. Writing this as in \eqref{eq:A_alpha1}, we distinguish between $\alpha_3 = \alpha$, $|\alpha| - 1\geq |\alpha_3|\geq\frac{N}{2}$, $|\alpha_3|<\frac{N}{2}$. In the second situation, we can exactly replicate the argument given for the plain energy bounds, except in the case when $D^{\alpha_3}$ does not involve the operator $\Gamma_{1,2}$, whence one of $D^{\alpha_{1,2}}$ involves this operator, and hence the product $D^{\alpha_1}(\sqrt{\cdots}) D^{\alpha_2}\eps_t$ cannot simply be placed into $L^\infty$ without incurring a loss. Assume first that
\[
D^{\alpha_2}\eps_t = D^{\alpha_2'}\Gamma_{1,2}\eps_t
\]
Then we exploit the gain in the Sobolev embedding due to our assumption of radiality: we distinguish between two cases, in each of which we have to exploit the null-structure:
\\
{\it{(i): $|r-t|<\frac{t}{10}$.}} Note that we have
\begin{equation}\label{eq:null1}
\eps_r^2 - \eps_t^2 = \frac{(\Gamma_1+\Gamma_2)\eps (\eps_r - \eps_t)}{r+t},
\end{equation}
and so we have\footnote{More precisely, the absolute value of the expression on the left is bounded by a linear combination of the absolute values of expressions like the one on the right, with $\alpha_3$ replaced by $\beta\leq \alpha_3$. This follows from the a priori bounds underlying our calculations. Our argument works as well for these more general expressions}
\begin{align*}
&\partial_{t,r}D^{\alpha_3}\big(\frac{1}{\sqrt{1+|\nabla_x\phi|^2 - \eps_t^2}} - \frac{1}{\sqrt{1+|\nabla_xQ|^2}}\big)\\&\sim \partial_{t,r}D^{\alpha_3}(Q_r\eps_r) + \partial_{t,r}D^{\alpha_3}\big[\frac{(\Gamma_1+\Gamma_2)\eps (\eps_r - \eps_t)}{r+t}\big]
\end{align*}
Since $|(\Gamma_1+\Gamma_2)\eps(t, r)|\lesssim K\kappa_0\la t\ra^{\delta}\la\log r\ra^{\frac{1}{2}}$ on the support\footnote{This follows from 
$|(\Gamma_1+\Gamma_2)\eps(t, r)|\lesssim \langle \log r\rangle^{\frac{1}{2}}\|\partial_r(\Gamma_1+\Gamma_2)\eps\|_{L^2_{rdr}}$}
 of the function under our assumption, as well as
\[
\|D^{\alpha}\Gamma^{\mu}\eps(t, r)\|_{L^\infty}\lesssim r^{-\frac{1}{2}}\|D^{\alpha}\Gamma^{\mu}\eps(t, r)\|_{H^1_{r\,dr}},
\]
we obtain obtain the bound
\begin{equation}\label{eq:nasty1}
\|\partial_{t,r}D^{\alpha_3}\big[\frac{(\Gamma_1+\Gamma_2)\eps (\eps_r - \eps_t)}{r+t}\big]\|_{L^2_{r\,dr}}\lesssim \frac{1}{t}(K\kappa_0)^2\la\log t\ra^{\frac{1}{2}}\la t\ra^{2\delta}
\end{equation}
and further
\begin{equation}\label{eq:simple1}
\|\partial_{t,r}D^{\alpha_3}(Q_r\eps_r)\|_{L^2_{r\,dr}}\lesssim (K\kappa_0)(\lambda - t)^{-1}\la t\ra^\delta
\end{equation}
We conclude that under our current assumptions, we can bound
\begin{align}
&\big|\int_0^t\int_0^\infty \chi^1(r,t)D^{\alpha_1}(\sqrt{1+|\nabla_x\phi|^2 - \eps_t^2})D^{\alpha_2}(\eps_t)\nonumber\\&\hspace{4cm}\times\partial_t D^{\alpha_3}\big(\frac{1}{\sqrt{1+|\nabla_x\phi|^2 - \eps_t^2}} - \frac{1}{\sqrt{1+|\nabla_xQ|^2}}\big)D^{\alpha}\eps_t\,rdrdt\big|\label{eq:longmess}
\\&\lesssim \int_0^t \|D^{\alpha_1}(\sqrt{1+|\nabla_x\phi|^2 - \eps_t^2})\|_{L_{rdr}^\infty}\|\chi^1(r,t)D^{\alpha_2}(\eps_t)\|_{L^\infty_{rdr}}\|\partial_t D^{\alpha_3}\big(\cdots\big)\|_{L^2_{rdr}}\|D^{\alpha}\eps_t\|_{L^2_{rdr}}\,dt\nonumber
\end{align}
where $\chi^1(r,t)$ localizes to the region $|r-t|<\frac{t}{10}$, $r\geq \lambda - t$. But we have
\[
\|\chi^1(r,t)D^{\alpha_2}(\eps_t)\|_{L^\infty_{rdr}}\lesssim t^{-\frac{1}{2}}\|D^{\alpha_2}(\eps_t)\|_{H^1_{rdr}}\lesssim (K\kappa_0) \la t\ra^{\delta-\frac{1}{2}},
\]
and so we can bound the preceding expression, using \eqref{eq:nasty1} as well as \eqref{eq:simple1} by
\[
\lesssim \int_0^t (K\kappa_0)^{3}(K\kappa_0\la s\ra^{4\delta - \frac{3}{2}}+ \la s\ra^{3\delta-\frac{1}{2}}(\lambda - s)^{-1})\la\log s\ra^{\frac{1}{2}}\,ds\lesssim (K\kappa_0)^3 \la t\ra^{2\delta},
\]
provided we choose $2\delta<\frac{1}{2}$.
\\

{\it{(ii): $|r-t|\geq \frac{t}{10}$.}} Here we use the identity
\begin{equation}\label{eq:null2}
\eps_r^2 - \eps_t^2 = \frac{(\Gamma_2\eps)^2 - (\Gamma_1\eps)^2}{r^2 - t^2},
\end{equation}
whence
\begin{align*}
&\partial_{t,r}D^{\alpha_3}\big(\frac{1}{\sqrt{1+|\nabla_x\phi|^2 - \eps_t^2}} - \frac{1}{\sqrt{1+|\nabla_xQ|^2}}\big)\\&\sim \partial_{t,r}D^{\alpha_3}(Q_r\eps_r) + \partial_{t,r}D^{\alpha_3}\big[\frac{(\Gamma_2\eps)^2 - (\Gamma_1\eps)^2}{r^2 - t^2}\big],
\end{align*}
and if $\chi^2$ localizes to $|r-t|\geq \frac{t}{10}, r\geq \lambda - t$, we obtain
\[
\|\chi^2\partial_{t,r}D^{\alpha_3}\big[\frac{(\Gamma_2\eps)^2 - (\Gamma_1\eps)^2}{r^2 - t^2}\big]\|_{L^2_{r\,dr}}
\lesssim (K\kappa_0)^2\la\log t\ra^{\frac{1}{2}} t^{2\delta-2}
\]
We conclude that we can bound the long expression \eqref{eq:longmess} by
\begin{align*}
&\lesssim (K\kappa_0)^4\int_0^t \la\log s\ra^{\frac{1}{2}}\la s\ra^{4\delta - 2}\,ds\\& + \int_0^t \|D^{\alpha_1}\big(\cdots\big)\|_{L^\infty}\|D^{\alpha_2}\eps_t\|_{L^\infty_{r\,dr}}\|\chi^2\partial_{t,r}D^{\alpha_3}(Q_r\eps_r)\|_{L^2_{r\,dr}}\|D^\alpha\eps_t\|_{L^2_{r\,dr}}\,dt
\end{align*}
In order to bound the contribution of the second integral expression, we have to exploit an extra gain in $t$. For this, note that
\[
\chi^2\eps_r = \chi^2\frac{r\Gamma_2\eps - t\Gamma_1\eps}{r^2 - t^2},
\]
 and hence we have
 \[
 \|\chi^2\partial_{t,r}D^{\alpha_3}\big(Q_r\eps_r\big)\|_{L^2_{r\,dr}}\lesssim (K\kappa_0)\la t\ra^{\delta - 1}
\]
In conjunction with\footnote{Recall that here the $D^{\alpha_2}$ involves an operator $\Gamma_{1,2}$ whence we cannot directly apply the dispersive estimate} 
\begin{equation}\label{eq:weakeps2}
|D^{\alpha_2}\eps_t|\lesssim (K\kappa_0)\langle t\rangle^{\delta}(\lambda - t)^{-\frac{1}{2}}
\end{equation}
We conclude that
\begin{align*}
&(K\kappa_0)^4\int_0^t \la\log s\ra^{\frac{1}{2}}\la s\ra^{4\delta - 2}\,ds\\& + \int_0^t \|D^{\alpha_1}\big(\cdots\big)\|_{L^\infty}\|D^{\alpha_2}\eps_t\|_{L^\infty_{r\,dr}}\|\chi^2\partial_{t,r}D^{\alpha_3}(Q_r\eps_r)\|_{L^2_{r\,dr}}\|D^\alpha\eps_t\|_{L^2_{r\,dr}}\,dt\\
&\lesssim  (K\kappa_0)^4\int_0^t \la\log s\ra^{\frac{1}{2}}\la s\ra^{4\delta - 2}\,ds + (K\kappa_0)^3\int_0^t\la s\ra^{3\delta - 1}(\lambda - s)^{-\frac{1}{2}}\,ds \ll (K\kappa_0)^2\la t\ra^{2\delta},
\end{align*}
provided $\delta<\frac{1}{2}$.
\\
The case when $D^{\alpha_1}\eps = \Gamma_{1,2}D^{\alpha_1'}\eps$ is handled analogously: as in \eqref{eq:boosted1} (recall that $|\alpha_1|\leq \frac{N}{2}$ under our current assumption $|\alpha_3|\geq \frac{N}{2}$) we find 
\begin{equation}\label{eq:alpha1Gamma}
\big|D^{\alpha_1}\big(\sqrt{1+|\nabla_x\phi|^2 - \eps_t^2}\big)\big|\lesssim [(\lambda - t)^{-1} + (K\kappa_0)\langle t\rangle^{-\frac{1}{2}}]\sum_{1\leq|\alpha|\leq |\alpha_1|+1}r^{-s(\alpha)}|D^{\alpha}\eps|
\end{equation}
where $s(\alpha)$ is defined as in \eqref{eq:boosted1}. Further, from \eqref{eq:unboosted1} and our assumption $|\alpha|-1\geq |\alpha_3|$, we have 
\[
\|\partial_{t,r}D^{\alpha_3}\big(\frac{1}{\sqrt{1+|\nabla_x\phi|^2 - \eps_t^2}} - \frac{1}{\sqrt{1+|\nabla_xQ|^2}}\big)\|_{L^2_{rdr}}\lesssim [(\lambda - t)^{-1} + (K\kappa_0)\langle t\rangle^{-\frac{1}{2}}](K\kappa_0)\langle t\rangle^{\delta}
\]
Combining these last two bounds with the simple $\|D^{\alpha_2}\eps_t\|_{L^\infty}\lesssim (K\kappa_0)\langle t\rangle^{-\frac{1}{2}}$, and inserting everything into \eqref{eq:longmess} (but without the cutoff $\chi^1$), we obtain 
\begin{align*}
|\eqref{eq:longmess}\,\text{without $\chi^1$}|&\lesssim \int_0^t\|D^{\alpha_1}(\ldots)(s, \cdot)\|_{L^\infty_{rdr}}\|D^{\alpha_2}\eps_s\|_{L^\infty}\|\partial_{s,r}D^{\alpha_3}(\ldots)\|_{L^2_{rdr}}\|D^{\alpha}\eps_s\|_{L^2_{rdr}}\,ds\\
&\lesssim (K\kappa_0)^4\int_0^t  [(\lambda - s)^{-1} + (K\kappa_0)\langle s\rangle^{-\frac{1}{2}}]^2\langle s\rangle^{3\delta-\frac{1}{2}}\,ds\ll (K\kappa_0)^2\langle t\rangle^{2\delta}
\end{align*}
provided that $\delta<\frac{1}{2}$.

This concludes the bootstrap for the contribution of $A_\alpha^1$ and the boosted energy, provided $|\alpha|-1\geq |\alpha_3|\geq \frac{N}{2}$.
\\
Now consider the case $|\alpha_3|<\frac{N}{2}$. First assume that $D^{\alpha_3} = D^{\alpha'_3}\Gamma_{1,2}$. If we have $\max|\alpha_{1,2}|\leq \frac{N}{2}$, then we have
\[
\| D^{\alpha_1}\big(\sqrt{1+|\nabla_x\phi|^2 - \eps_t^2}\big)D^{\alpha_2}\eps_t\|_{L^\infty_{r\,dr}}
\lesssim (K\kappa_0)\la t\ra^{-\frac{1}{2}},
\]
and so we bound \eqref{eq:longmess} without the localizer $\chi^1$ by
\[
\lesssim \int_0^t (K\kappa_0)^3\la s\ra^{-\frac{1}{2}}[\lambda - s)^{-1} + (K\kappa_0)\la s\ra^{-\frac{1}{2}}]\la s\ra^{2\delta}\,ds\ll (K\kappa_0)\la t\ra^{2\delta}
\]
where we have used the bound \eqref{eq:boosted1}. 
Next, if $|\alpha_2|>\frac{N}{2}$ and $D^{\alpha_3}$ is as before, we split into the regimes $|r-t|<t^{\delta_1}, \delta_1\ll 1$, and the complement. Note that necessarily $|\alpha_2|\leq N-2$, whence in the first situation ($|r-t|<t^{\delta_1}$) we have (using Sobolev)
\[
|D^{\alpha_2}\eps_t|\lesssim (K\kappa_0)\la t\ra^{\delta - \frac{1}{2}}, |\partial_t D^{\alpha_3}\big(\ldots\big)|\lesssim (K\kappa_0) [(\lambda - t)^{-1}+(K\kappa_0)\langle t\rangle^{-\frac{1}{2}}]\langle t\rangle^{\delta}(\lambda - t)^{-\frac{1}{2}}
\]
where for the second bound we have used \eqref{eq:boosted1}. 
It follows that we control \eqref{eq:longmess} under the  restriction $|r-t|<t^{\delta_1}$ by (using Cauchy Schwarz with respect to $\mu$)
\[
\lesssim \int_0^t (K\kappa_0) [(\lambda - s)^{-1}+(K\kappa_0)\langle s\rangle^{-\frac{1}{2}}]\langle s\rangle^{3\delta+\frac{\delta_1}{2}-\frac{1}{2}}(\lambda - s)^{-\frac{1}{2}}\,ds
\ll (K\kappa_0)^2\la t\ra^{2\delta},
\]
provided $\delta+\frac{\delta_1}{2}<\frac{1}{2}$. If on the other hand $|r-t|\geq t^{\delta_1}$, we write
\[
D^{\alpha_2}\eps_t = \frac{t\Gamma_2 D^{\alpha_2}\eps - r\Gamma_1 D^{\alpha_2}\eps}{t^2 - r^2},
\]
and further
\[
|\partial_t D^{\alpha_3}\big(\ldots\big)|\lesssim (K\kappa_0) [(\lambda - t)^{-1}+(K\kappa_0)\langle t\rangle^{-\frac{1}{2}}]\langle t\rangle^{\delta}(\lambda - t)^{-\frac{1}{2}}
\]
and so we bound \eqref{eq:longmess} in this situation by
\begin{align*}
&\lesssim \int_0^t\| D^{\alpha_1}(\ldots)\|_{L^\infty_{r\,dr}}\|\chi_{|r-t|\geq t^{\delta_1}}\frac{t\Gamma_2 D^{\alpha_2}\eps - r\Gamma_1 D^{\alpha}\eps}{t^2 - r^2}\|_{L^2_{r\,dr}}\|\partial_t D^{\alpha_3}\big(\ldots\big)\|_{L^\infty_{r\,dr}}\|D^{\alpha}\eps_t\|_{L^2_{r\,dr}}\,dt\\
&\lesssim (K\kappa_0)^3\int_0^t\langle s\rangle^{2\delta-\delta_1}[(\lambda - s)^{-1}+(K\kappa_0)\langle s\rangle^{-\frac{1}{2}}]\langle s\rangle^{\delta}(\lambda - s)^{-\frac{1}{2}}
\,ds\\
&\ll (K\kappa_0)^2\la t\ra^{2\delta},
 \end{align*}
 provided $\delta <\delta_1$.
 The argument for the case $|\alpha_1|\geq \frac{N}{2}$ is similar.
 Next, consider the case when the derivative $\Gamma_{1,2}$ is either in $D^{\alpha_1}$ or $D^{\alpha_2}$. Then in light of \eqref{eq:alpha1Gamma} we have the bound
 \[
 \|D^{\alpha_1}\big(\sqrt{1+|\nabla_x\phi|^2 - \eps_t^2}\big)D^{\alpha_2}\eps_t\|_{L^2_{rdr}}\lesssim (K\kappa_0)\langle t\rangle^{\delta}
 \]
 while in light of equation \eqref{eq:boosted1} we also have
 \[
 \|\partial_t D^{\alpha_3}\big(\ldots\big)\|_{L^\infty_{r\,dr}}\lesssim (K\kappa_0)^2 \la t\ra^{-1}+(K\kappa_0)\la t\ra^{-\frac{1}{2}}(\lambda - t)^{-1},
 \]
 whence in this situation the expression \eqref{eq:longmess} is bounded by  \[
 \lesssim \int_0^t (K\kappa_0)\la s\ra^{2\delta}[ (K\kappa_0)^2 \la s\ra^{-1}+(K\kappa_0)\la s\ra^{-\frac{1}{2}}(\lambda - s)^{-1}]\,ds\ll (K\kappa_0)\la t\ra^{2\delta},\,0\leq t\leq\lambda - C_1
 \]

Finally, in case $\alpha_3 = \alpha$, it suffices to consider the case when all derivatives in $D^{\alpha_3}$ fall on a second derivative term, i. e. the first two terms of \eqref{eq:1} (the remaining cases are treated as before), and for these one proceeds exactly as after \eqref{eq:1}, using integration by parts. This concludes the contribution of $A_\alpha^1$.
\\

{\it{Contribution of $A_\alpha^2$.}} Next, we treat the contribution of the two terms in \eqref{eq:enerror1}, again in the situation where one of $D^{\alpha_{1,2,3}}$ involves a vector fields $\Gamma_{1,2}$. We may assume that $|\alpha_3|<|\alpha|$ or that not all derivatives fall on $\eps_{rr}$ since else one replicates the integration by parts argument from the energy bounds in \eqref{eq:2}.
\\
{\it{Start with $|\alpha_1|\leq \frac{N}{2}$.}}
First, assume that one of $D^{\alpha_{2,3}}$ involves $\Gamma_{1,2}$. Then we have
\[
Y_1(t): = \|D^{\alpha_2}\big(O(\frac{1}{r^2})\big) D^{\alpha_3}\big(\frac{\eps_{rr} + O(\frac{1}{r})\eps_r}{(1+|\nabla_x\phi|^2 - \eps_t^2)^{\frac{3}{2}}}\big)\|_{L^2_{rdr}}\lesssim (\lambda - t)^{-2}K\kappa_0 \la t\ra^{\delta},
\]
where one uses relations like after \eqref{eq:boosted2} of the form 
\[
\Gamma_1\big(O(\frac{1}{r^2})\big)\tilde{\eps}_{rr} =  O(\frac{1}{r^3})\big(\Gamma_1\tilde{\eps}_r - r\tilde{\eps}_{tr}\big),\,\Gamma_2\big(O(\frac{1}{r^2})\big)\tilde{\eps}_{rr} =  \big(O(\frac{1}{r^2})\big)\tilde{\eps}_{rr} 
\]

Also, we find
\[
Y_2(t): = \|D^{\alpha_2}\big(O(\frac{1}{r})\big)D^{\alpha_3}\big(\frac{\eps_r\eps_{rr} - \eps_t\eps_{tr}}{(1+|\nabla_x\phi|^2 - \eps_t^2)^{\frac{3}{2}}}\big)\|_{L^2_{rdr}}\lesssim (K\kappa_0)^2(\lambda - t)^{-1}t^{-\frac{1}{2}}\la t\ra^{2\delta}
\]
To see this, write
\[
|D^{\alpha}(\eps_r\eps_{rr} - \eps_t\eps_{tr})| \lesssim \sum_{\beta_1+\beta_2 = \alpha}|D^{\beta_1}\eps_rD^{\beta_2}\eps_{rr} - D^{\beta_1}\eps_t D^{\beta_2}\eps_{tr}|
\]
If $|\alpha|<N$, at most one of $|\beta_1|$, $|\beta_2|$ is $>\frac{N}{2}$. If for this multi-index $\beta_{1,2}$ we have that $D^{\beta_{1,2}}$ also involves $\Gamma_{1,2}$ exactly once, and the other operator $D^{\beta_{2,1}}$ does not, then under our assumptions
\[
\|D^{\beta_1}\eps_rD^{\beta_2}\eps_{rr} - D^{\beta_1}\eps_t D^{\beta_2}\eps_{tr}\|_{L^2_{r\,dr}}
\lesssim (K\kappa_0)^2\la t\ra^{\delta-\frac{1}{2}}
\]
On the other hand, if the unique operator $D^{\beta_{1,2}}$ with $|\beta_{1,2}|>\frac{N}{2}$ does not involve $\Gamma_{1,2}$, but the other operator does, then assuming say $|\beta_2|>\frac{N}{2}$, we get (using the improved radial Sobolev embedding)
\begin{align*}
&\|D^{\beta_1}\eps_rD^{\beta_2}\eps_{rr} - D^{\beta_1}\eps_t D^{\beta_2}\eps_{tr}\|_{L^2_{r\,dr}}
\\
&\lesssim \|\chi_{r\geq t}D^{\beta_1}\eps_{r,t}\|_{L^\infty}\|D^{\beta_2}(\eps_{r,t})_{r}\|_{L^2_{r\,dr}}\\&+\|\chi_{r<t} \la r\ra^{\frac{1}{2}}D^{\beta_1}\eps_{r,t}\|_{L^\infty}\frac{1}{t}\|\chi_{r<t}\frac{\Gamma_1D^{\beta_2}\eps_{r,t}- rD^{\beta_2}(\eps_{r,t})_{t}}{\la r\ra^{\frac{1}{2}}}\|_{L^2_{r\,dr}}\\
&\lesssim (K\kappa_0)^2t^{-\frac{1}{2}}\la t\ra^{2\delta}
\end{align*}
The case $|\beta_1|>\frac{N}{2}$ is similar, and this establishes the bound on $Y_2(t)$. Since we also assumed $|\alpha_1|\leq \frac{N}{2}$, we have
\[
Y_3(t): = \|D^{\alpha_1}(\sqrt{1+|\nabla_x\phi|^2 - \eps_t^2})\|_{L^\infty_{r\,dr}}\lesssim 1,
\]
whence we obtain the bound
\begin{align*}
&\int_0^t Y_3(s)(Y_1(s) + Y_2(s))\|D^{\alpha}\eps_t\|_{L^2_{r\,dr}}\,ds\\
&\lesssim \int_0^t (K\kappa_0)^2(\lambda - s)^{-1}\la s\ra^{2\delta}[(\lambda - s)^{-1} + \la s\ra^{\delta}s^{-\frac{1}{2}}]\,ds\\
&\ll (K\kappa_0)^2\la t\ra^{2\delta},\,0\leq t\leq \lambda - C_1
\end{align*}
provided $C_1\gg 1$.
\\
{\it{Next, assuming $|\alpha_1|>\frac{N}{2}$}} while still assuming $\Gamma_{1,2}$ to occur in $D^{\alpha_{2,3}}$, we use
\begin{align*}
Y_1(t)\lesssim &(\lambda - t)^{-2}\sum_{|\beta|\leq |\alpha_3|}(\|\Gamma_{1,2}D^{\beta}\eps_r\|_{L^2_{r\,dr}} +\|D^{\beta}(\eps_{r,t})_{r}\|_{L^2_{r\,dr}})\\
&+(\lambda - t)^{-1}\sum_{|\beta|\leq |\alpha_3|}(\|\frac{\Gamma_{1,2}D^{\beta}\eps}{r}\|_{L^2_{r\,dr}} +\|\frac{D^{\beta}\eps_{r,t}}{r}\|_{L^2_{r\,dr}})\\
&\lesssim (K\kappa_0)(\lambda - t)^{-1} \la t\ra^{\delta}
\end{align*}
as well as
\[
Y_2(t)\lesssim (K\kappa_0)^2(\lambda - t)^{-1}\la t\ra^{\delta - \frac{1}{2}}.
\]
On the other hand, for the term $Y_3(t)$, we get 
\begin{align*}
\|D^{\alpha_1}(\sqrt{1+|\nabla_x\phi|^2 - \eps_t^2})\|_{L^\infty_{rdr}}&\sim \|D^{\alpha_1}\big(\frac{1}{r}\eps_r + \eps_r^2 - \eps_t^2\big)\|_{L^\infty_{rdr}}\\
&\lesssim \|D^{\alpha_1}\big(\frac{1}{rt}(\Gamma_1\eps - r\eps_t)\big)\|_{L^\infty_{rdr}} + \|D^{\alpha_1}\big(\eps_r^2 - \eps_t^2\big)\|_{L^\infty_{rdr}}\\
&\lesssim (K\kappa_0)\langle t\rangle^{\delta - 1} + (K\kappa_0)^2\langle t\rangle^{\delta - \frac{1}{2}}
\end{align*}
where we used the fact that $|\alpha_1|<|\alpha|$ under our current assumptions, as well as the radial Sobolev embedding $H^1_{r>1}\subset L^\infty$. 
Hence in the present case, we get the bound
\begin{align*}
&\int_0^t Y_3(s)(Y_1(s) + Y_2(s))\|D^{\alpha}\eps_t\|_{L^2_{r\,dr}}\,ds\\
&\lesssim (K\kappa_0)^2\int_0^t \langle s\rangle^{3\delta}[(\lambda - s)^{-2}+(K\kappa_0)(\lambda - s)^{-1}][\langle s\rangle^{\delta-1}+(K\kappa_0)\langle s\rangle^{\delta - \frac{1}{2}}]\,ds                        \ll (K\kappa_0)^2\la t\ra^{2\delta}
\end{align*}
provided $2\delta<\frac{1}{2}$ and $0\leq t\leq \lambda - C_1$ with $C_1\gg 1$. 
Finally, the case when $\Gamma_{1,2}$ is contained in $D^{\alpha_1}$ is handled similarly and omitted.
\\

{\it{Contribution of $A_{\alpha}^{3,4}$.}} These are handled like $A_{\alpha}^{2}$ in the case $|\alpha_3|<|\alpha|$.
\\

{\it{Contribution of the commutator term $[D^{\alpha}, \Box]$.}} Write $D^{\alpha} = D^{\alpha'}\Gamma_{1,2}$, $|\alpha'|\leq N-2$ with $D^{\alpha'} = \partial_t^{\alpha_1'}\partial_r^{\alpha_2'}$. Then we have
\[
[D^{\alpha'}\Gamma, \Box] = D^{\alpha'}[\Gamma, \Box] + [D^{\alpha'}, \Box]\Gamma
\]
and further
\[
[\Gamma_2, \Box] = -2\Box,\,[\Gamma_1, \Box] = \frac{1}{r^2}\Gamma_1,
\]
whence we have
\[
[D^{\alpha'}\Gamma_2, \Box] = -2D^{\alpha'}\Box + \sum_{1\leq |\beta|\leq |\alpha|-1}F_{\beta}(r)\partial_r^{\beta}\Gamma_2
\]
\[
[D^{\alpha'}\Gamma_1, \Box] = D^{\alpha'}\frac{1}{r^2}\Gamma_1+ \sum_{1\leq |\beta|\leq |\alpha|-1}F_{\beta}(r)\partial_r^{\beta}\Gamma_1
\]
where we have $F_{\beta}(r) = O(r^{-2-|\alpha'-\beta|})$.Then we need to estimate
\[
\int_0^t\int_0^\infty \big[-D^{\alpha'}\frac{1}{r^2}\Gamma_1+ \sum_{1\leq |\beta|\leq |\alpha|-1}F_{\beta}(r)\partial_r^{\beta}\Gamma_1\big]\eps(D^{\alpha'}\Gamma_1\eps)_{t}r\,drdt
\]
\[
\int_0^t\int_0^\infty\big[2D^{\alpha'}\Box + \sum_{1\leq |\beta|\leq |\alpha|-1}F_{\beta}(r)\partial_r^{\beta}\Gamma_2\big]\eps(D^{\alpha'}\Gamma_2\eps)_{t}r\,drdt
\]
We easily obtain ($\Gamma = \Gamma_{1,2}$)
\begin{align*}
&\big|\int_0^t\int_0^\infty\big[\sum_{1\leq |\beta|\leq |\alpha|-1}F_{\beta}(r)\partial_r^{\beta}\Gamma\big]\eps(D^{\alpha'}\Gamma\eps)_{t}r\,drdt\big|\\
&\lesssim \int_0^t (K\kappa_0)^2(\lambda - s)^{-2}\la s\ra^{2\delta}\,ds\ll (K\kappa_0)^2\la t\ra^{2\delta},\,0\leq t\leq \lambda - C_1,
\end{align*}
provided $C_1\gg 1$. Next, we estimate
\begin{align*}
&\big|\int_0^t\int_0^\infty\big[-D^{\alpha'}\frac{1}{r^2}\Gamma_1\eps\big](D^{\alpha'}\Gamma_1\eps)_{t}r\,drdt\big|\\
\end{align*}
Here we have to be careful to avoid a potential logarithmic loss. Thus write
\[
D^{\alpha'}\frac{1}{r^2}\Gamma_1\eps = \frac{1}{r^2}D^{\alpha'}\Gamma_1\eps + O(\frac{1}{r^3}\sum_{|\gamma|<|\alpha'|}|D^\gamma\Gamma_1\eps|)
\]
For the contribution of the error term, we have (recalling the constraints for the support of the integrand)
\begin{align*}
&\big|\int_0^t\int_0^\infty O(\frac{1}{r^3}\sum_{|\gamma|<|\alpha'|}|D^\gamma\Gamma_1\eps|)
(D^{\alpha'}\Gamma_1\eps)_{t}r\,drdt\big|\\
&\lesssim \int_0^t(\lambda - s)^{-2}\|\sum_{|\gamma|<|\alpha'|}|\frac{D^\gamma\Gamma_1\eps}{r}|\|_{L^2_{r\,dr}}\|(D^{\alpha'}\Gamma_1\eps)_{s}\|_{L^2_{r\,dr}}\,ds\\
&\lesssim  (K\kappa_0)^2\int_0^t(\lambda - s)^{-2}\la s\ra^{2\delta}\,ds\ll (K\kappa_0)^2\la t\ra^{2\delta},\,0\leq t\leq \lambda - C_1,\,C_1\gg 1.
\end{align*}
For the leading term above, we find
\begin{align*}
&\int_0^t\int_0^\infty\big[-\frac{1}{r^2}D^{\alpha'}\Gamma_1\eps\big](D^{\alpha'}\Gamma_1\eps)_{t}r\,drdt\big|\\
&=-\frac{1}{2}\int_0^\infty\big(\frac{D^{\alpha'}\Gamma_1\eps(t, \cdot)}{r}\big)^2 r\,dr + \frac{1}{2}\int_0^\infty\big(\frac{D^{\alpha'}\Gamma_1\eps(0, \cdot)}{r}\big)^2 r\,dr\\
&\leq \frac{1}{2}\int_0^\infty\big(\frac{D^{\alpha'}\Gamma_1\eps(0, \cdot)}{r}\big)^2 r\,dr
\end{align*}
and we have
\[
\frac{1}{2}\int_0^\infty\big(\frac{D^{\alpha'}\Gamma_1\eps(0, \cdot)}{r}\big)^2 r\,dr\ll (K\kappa_0)^2
\]
upon choosing $K$ suitably.
\\
Finally, the contribution of the term $-2D^{\alpha'}\Box$, i. e. the expression
\[
\int_0^t \int_0^\infty -2D^{\alpha'}\Box\eps \big(D^{\alpha'}\Gamma_2\eps\big)_t r\,drdt
\]
is handled precisely like the contributions of $A_\alpha^1$ - $A_\alpha^4$, using the equation of $\eps$. This concludes the bootstrap for the boosted energy bounds.
\\

{\bf{The dispersive estimate}}: Finally we improve the bound
\[
\sum_{1\leq\alpha\leq \frac{N}{2}+2}\|\langle t\rangle^{\frac{1}{2}}\partial_{t,r}^{\alpha}\eps(t, \cdot)\|_{L_{t,r}^\infty([0,T]\times \R_{+})}\leq K\kappa_0
\]
Due to the already bootstrapped energy bounds, we may assume $t\gg 1$.
Our point of departure is again \eqref{eq:bigmess}, with $|\alpha|\geq 1$. Denoting the right hand side by $F_\alpha(t, r)$, we pass to a one-dimensional formulation via $r^{\frac{1}{2}}D^{\alpha}\eps(t, r) =:\tilde{\eps}_\alpha(t, r)$.
Then we find that
\[
\tilde{\Box}\tilde{\eps}_{\alpha} + r^{\frac{1}{2}}[D^{\alpha}, \Box]\eps = r^{\frac{1}{2}}F_\alpha
\]
Here we have introduced $\tilde{\Box} = \partial_t^2 - \partial_r^2$. We can solve this problem by invoking the odd extension of all functions (with respect to $r$) to $(-\infty, \infty)$ and using the standard  d'Alembert's solution. We then obtain
\[
\tilde{\eps}_{\alpha} = \int_0^t\int_{|r-(t-\tilde{t})}^{r+(t-\tilde{t})}\big[- \mu^{\frac{1}{2}}[D^{\alpha}, \Box]\eps(\tilde{t}, \mu) + \mu^{\frac{1}{2}}F_\alpha(\tilde{t}, \mu)\,d\mu d\tilde{t}  + \tilde{\eps}_{\alpha, \text{free}}
\]
where $\tilde{\Box}\tilde{\eps}_{\alpha, \text{free}} = 0$, $\tilde{\eps}_{\alpha, \text{free}}[0] =  (\tilde{\eps}_\alpha(0, \cdot), \partial_t\tilde{\eps}_\alpha(0, \cdot))$. Then the bound
\[
\|\tilde{\eps}_{\alpha, \text{free}}(t, \cdot)\|_{L^\infty}\ll K\kappa_0\la t\ra^{-\frac{1}{2}}
\]
follows from easily from the d'Alembert parametrix, and
we thus need to show that under our assumptions, we have the bound
\[
r^{-\frac{1}{2}}\int_0^t\int_{|r-(t-\tilde{t})|}^{r+(t-\tilde{t})}\big[-\mu^{\frac{1}{2}}[D^{\alpha}, \Box]\eps(\tilde{t}, \mu) + \mu^{\frac{1}{2}}F_\alpha(\tilde{t}, \mu)\,d\mu d\tilde{t}\big] \ll \la t\ra^{-\frac{1}{2}}K\kappa_0
\]
We again treat the various ingredients forming $F_\alpha$ and the first term in the integrand separately.
\\

{\it{(i): The contribution of the term $-\mu^{\frac{1}{2}}[D^{\alpha}, \Box]\eps(\tilde{t}, \mu)$.}}
We write $[D^{\alpha}, \Box]\eps(\tilde{t}, \mu) = \sum_{1\leq|\beta|\leq |\alpha|}F_\beta(\mu)\partial^\beta\eps(\tilde{t}, \mu)$ where $F_\beta(\mu) = O(\frac{1}{\mu^{2+|\alpha|-|\beta|}})$.
Then we distinguish between the following cases:
\\

{\it{(i1): $\tilde{t}\ll t$, $r\ll t$.}} Note that then $|r\pm(t-\tilde{t})|\sim t$, and due to Huyghen's principle, we also get $t\gg 1$ on the support of $\eps$. Then we estimate
\begin{align*}
&r^{-\frac{1}{2}}\big|\int_0^t\int_{|r-(t-\tilde{t})|}^{r+(t-\tilde{t})}\chi_{\tilde{t}\ll t}\sum_{1\leq|\beta|\leq |\alpha|}\mu^{\frac{1}{2}}F_\beta(\mu)\partial^\beta\eps(\tilde{t}, \mu)\,d\mu d\tilde{t}\big|\\
&\lesssim \sum_{1\leq|\beta|\leq |\alpha|}\int_0^t \chi_{\tilde{t}\ll t}\la t\ra^{-2}\|\partial^\beta\eps(\tilde{t}, \cdot)\|_{L^2_{\mu d\mu}}\,d\tilde{t}\lesssim K\kappa_0\la t\ra^{-1+\delta}\ll K\kappa_0\la t\ra^{-\frac{1}{2}}
\end{align*}
We have used here the Cauchy-Schwarz inequality with respect to the $\mu$-integral.
\\

{\it{(i2): $\tilde{t}\ll t$, $r\gtrsim t$.}} Here the $\la t\ra^{-\frac{1}{2}}$--decay comes from the $r^{-\frac{1}{2}}$-factor outside. Estimate
\begin{align*}
&r^{-\frac{1}{2}}\big|\int_0^t\int_{|r-(t-\tilde{t})|}^{r+(t-\tilde{t})}\chi_{\tilde{t}\ll t}\sum_{1\leq|\beta|\leq |\alpha}\mu^{\frac{1}{2}}F_\beta(\mu)\partial^\beta\eps(\tilde{t}, \mu)\,d\mu d\tilde{t}\big|\\
&\lesssim \la t\ra^{-\frac{1}{2}}\sum_{1\leq|\beta|\leq |\alpha}\int_0^t \chi_{\tilde{t}\ll t}(\lambda-\tilde{t})^{-\frac{3}{2}}\|\partial^\beta\eps(\tilde{t},\cdot)\|_{L^2_{\mu d\mu}}\,d\tilde{t}
\end{align*}
In the last inequality, we have used that $\mu\geq \lambda - \tilde{t}$ on the support of $\eps(\tilde{t}, \mu)$. Since $\lambda - \tilde{t}\gg \tilde{t}$ for $\tilde{t}\ll t$, we can bound the last expression by
\[
\lesssim  \la t\ra^{-\frac{1}{2}}\int_0^t \chi_{\tilde{t}\ll t}K\kappa_0(\lambda-\tilde{t})^{-\frac{3}{2}+\delta}\,d\tilde{t}
\ll K\kappa_0\la t\ra^{-\frac{1}{2}},
\]
 as desired.
 \\

 {\it{(i3): $\tilde{t}\gtrsim t$.}} in this case, we have to exploit control over the vector fields $\Gamma_{1,2}\eps$. Simply write (for $|\beta|\geq 1$)
 \[
 \partial^{\beta}\eps = \partial^{\gamma}\partial_{t,r}\eps = \partial^{\gamma}\big(\frac{1}{t}[\Gamma_{2,1}\eps - r\eps_{r,t}]\big)
 \]
whence
\[
|\mu^{\frac{1}{2}}F_\beta(\mu)\partial^{\beta}\eps|\lesssim \mu^{-2}\big(\frac{1}{t}\sum_{|\tilde{\gamma}|\leq |\gamma|}\mu^{\frac{1}{2}}|\partial^{\tilde{\gamma}}\Gamma_{1,2}\eps| + \frac{\mu}{t}\sum_{1\leq |\tilde{\gamma}|\leq |\beta|}\mu^{\frac{1}{2}}|\partial^{\tilde{\gamma}}\eps|\big)
\]
One then estimates
\begin{align*}
&r^{-\frac{1}{2}}\big|\int_0^t\int_{|r-(t-\tilde{t})|}^{r+(t-\tilde{t})}\chi_{\tilde{t}\gtrsim t}\sum_{1\leq|\beta|\leq |\alpha|}\mu^{\frac{1}{2}}F_\beta(\mu)\partial^\beta\eps(\tilde{t}, \mu)\,d\mu d\tilde{t}\big|\\
&\lesssim t^{-1}\sum_{|\tilde{\gamma}|\leq |\beta|-1}\int_0^t \chi_{\tilde{t}\gtrsim t}(\lambda - \tilde{t})^{-1}\|\frac{\partial^{\tilde{\gamma}}\Gamma_{1,2}\eps(\tilde{t},\cdot)}{\mu}\|_{L^2_{\mu d\mu}}\,d\tilde{t}\\
&+t^{-1}\sum_{1\leq|\tilde{\gamma}|\leq |\beta|}\int_0^t \chi_{\tilde{t}\gtrsim t}(\lambda - \tilde{t})^{-1}\|\partial^{\tilde{\gamma}}\eps(\tilde{t},\cdot)\|_{L^2_{\mu d\mu}}\,d\tilde{t}\\
&\ll(K\kappa_0)\la \log t\ra t^{-1+\delta}\lesssim (K\kappa_0) t^{-\frac{1}{2}}
\end{align*}
Here we have again invoked the Cauchy-Schwarz inequality with respect to the $\mu$-integration, as well as a simple version of Hardy's inequality as well as the already bootstrapped energy bounds.
\\
This concludes the case (i).
\\

Next, we distinguish between the contribution of the null-form, \eqref{eq:ennullform}, and the remaining terms in $F_\alpha(\tilde{t}, \mu)$, which are treated like the term (i).

{\it{(ii): The contribution  of the source terms $F_\alpha(\tilde{t}, \mu)$; the null-form.}} \\
We use the following identity for the null-form:
\begin{equation}\label{eq:nullform}
\partial_t f \partial_t g - \partial_r f \partial_r g = \frac{\Gamma_2f\Gamma_2g - \Gamma_1f\Gamma_1g}{t^2 - r^2}
\end{equation}
Thus we write
\begin{align}
&-\eps_t\partial_t(\frac{1}{\sqrt{1+|\nabla_x\phi|^2 - \eps_t^2}} - \frac{1}{\sqrt{1+|\nabla_xQ|^2}})+\sum_{i=1,2}\eps_{x_i}\partial_{x_i}(\frac{1}{\sqrt{1+|\nabla_x\phi|^2 - \eps_t^2}} - \frac{1}{\sqrt{1+|\nabla_xQ|^2}})\label{eq:fullmess}\\
&=\frac{\sum_{j=1,2}(-1)^{j+1}\Gamma_j\eps \Gamma_j(\frac{1}{\sqrt{1+|\nabla_x\phi|^2 - \eps_t^2}} - \frac{1}{\sqrt{1+|\nabla_xQ|^2}})}{t^2 - r^2}\nonumber
\end{align}
Note that
\begin{align*}
&D^{\alpha}\Gamma_2(\frac{1}{\sqrt{1+|\nabla_x\phi|^2 - \eps_t^2}} - \frac{1}{\sqrt{1+|\nabla_xQ|^2}})\\& = O(\frac{1}{r})\sum_{|\beta|\leq |\alpha|+1}|D^{\beta}\eps_r| + O(\sum_{|\beta|\leq |\alpha|}|D^{\beta}\Gamma_2\eps_{t,r}|\sum_{|\beta|\leq |\alpha|}|D^{\beta}\eps_{t,r}|)
\end{align*}
while also (in the first sum the operator $D^{\beta}$ may involve at most one operator $\Gamma_{1,2}$ and none in the second and third)
\begin{align*}
&D^{\alpha}\Gamma_1(\frac{1}{\sqrt{1+|\nabla_x\phi|^2 - \eps_t^2}} - \frac{1}{\sqrt{1+|\nabla_xQ|^2}})\\& =
O(\frac{1}{r})\sum_{|\beta|\leq |\alpha|+2}r^{-s(\beta)}|D^{\beta}\eps|+ O(\sum_{|\beta|\leq |\alpha|}|D^{\beta}\Gamma_1\eps_{t,r}|\sum_{|\beta|\leq |\alpha|}|D^{\beta}\eps_{t,r}|)
\end{align*}
where we have $s(\beta) = 1$ if $D^{\beta} = \Gamma_{1,2}$ and $s(\beta) = 0$ otherwise. Introduce the quantity
\[
G_\alpha =  O(\frac{1}{r})\sum_{|\beta|\leq |\alpha|+1}|D^{\beta}\eps_r| + O(\frac{1}{r})\sum_{|\beta|\leq |\alpha|+2}r^{-s(\beta)}|D^{\beta}\eps|
\]
{\it{(iia): bounding the integral}}
\[
r^{-\frac{1}{2}}\sum_{\alpha_1+\alpha_2=\alpha}\big|\int_0^t\int_{|r-(t-\tilde{t})|}^{r+(t-\tilde{t})}\mu^{\frac{1}{2}}\frac{|D^{\alpha_1}\Gamma\eps|\, G_{\alpha_2}}{\tilde{t}^2 - \mu^2}\,d\mu d\tilde{t}\big|,\,\Gamma = \Gamma_{1,2}
\]
We distinguish between the following cases:
\\
{\it{(iia.1): $\tilde{t}\ll t$.}} Here we have $|\mu^2 - \tilde{t}^2|\gtrsim t^2$, whence for $|\alpha|\leq \frac{N}{2}+2$, $\alpha = \alpha_1+\alpha_2$, we get
\begin{align*}
&r^{-\frac{1}{2}}\big|\int_0^t\int_{|r-(t-\tilde{t})|}^{r+(t-\tilde{t})}\chi_{\tilde{t}\ll t}\mu^{\frac{1}{2}}\frac{|D^{\alpha_1}\Gamma\eps|\, G_{\alpha_2}}{-\tilde{t}^2 + \mu^2}\,d\mu d\tilde{t}\big|\\&\lesssim \la t\ra^{-1} \frac{\la\log t\ra^{\frac{1}{2}}}{\la t\ra}\int_0^t\|\frac{D^{\alpha_1}\Gamma\eps(\tilde{t}, \cdot)}{\la \log \mu\ra^{\frac{1}{2}}}\|_{L^\infty_{\mu d\mu}}\|G_{\alpha_2}(\tilde{t}, \cdot)\|_{L^2_{\mu d\mu}}\,d\tilde{t}\\
&\lesssim (K\kappa_0)^2\la t\ra^{-1+2\delta}\la \log t\ra^{\frac{1}{2}}\ll (K\kappa_0)\la t\ra^{-\frac{1}{2}},\,\text{provided}\,\,2\delta<\frac{1}{2}
\end{align*}
Here we have used Cauchy-Schwarz with respect to $\mu$ and the bound
\[
\|\frac{D^{\alpha_1}\Gamma\eps(\tilde{t}, \cdot)}{\la\log\mu\ra^{\frac{1}{2}}}\|_{L^\infty_{\mu d\mu}}\lesssim \|D^{\alpha_1}\Gamma\eps(\tilde{t}, \cdot)\|_{\dot{H}^1_{\mu\,d\mu}}
\]
which comes from the support properties of $\eps$ and our current restrictions on $r, t$.
\\

{\it{(iia.2): $\tilde{t}\gtrsim t$.}} We observe that if further $|\tilde{t}^2 - \mu^2|\sim \tilde{t}^2\sim t^2$, one argues exactly as in case (iia.1). If $|\tilde{t}^2 - \mu^2|\ll \tilde{t}^2$, then $\mu\sim \tilde{t}\sim t$. Thus if we further restrict to $|\tilde{t}^2 - \mu^2|\geq \tilde{t}^{1-\delta_1}$, we get, using
\[
\|G_{\alpha_2}\|_{L^2_{\mu\,d\mu}}\lesssim (K\kappa_0)t^{\delta-1}
\]
and introducing a cutoff $\chi^2$ to implement the above restrictions on $\mu, \tilde{t}$
\begin{align*}
&r^{-\frac{1}{2}}\big|\int_0^t\int_{|r-(t-\tilde{t})|}^{r+(t-\tilde{t})}\chi^2\mu^{\frac{1}{2}}\frac{|D^{\alpha_1}\Gamma\eps|\, G_{\alpha_2}}{-\tilde{t}^2 + \mu^2}\,d\mu d\tilde{t}\big|\\&
\lesssim (K\kappa_0)^2t^{-(2-2\delta-\delta_1)}\int_0^t\la \log t\ra^{\frac{1}{2}}d\tilde{t}\ll (K\kappa_0)\la \log t\ra^{\frac{1}{2}}t^{-(1-2\delta - \delta_1)}\lesssim (K\kappa_0)t^{-\frac{1}{2}}
\end{align*}
provided $2\delta + \delta_1<\frac{1}{2}$.\\
Thus we may further restrict to $|\tilde{t}^2 - \mu^2|<\tilde{t}^{1-\delta_1}, \tilde{t}\gtrsim t$, which we do via a multiplier $\chi^3$.
Introduce the quantity
\begin{align*}
H_{\alpha}: &= D^{\alpha}\partial_{t,r}\big(\frac{1}{\sqrt{1+|\nabla_x\phi|^2 - \eps_t^2}} - \frac{1}{\sqrt{1+|\nabla_x Q|^2}}\big)\\
& = O(\frac{1}{\mu})\sum_{|\beta|\leq |\alpha|+1}D^{\beta}\eps_\mu + O([\sum_{|\beta|\leq \frac{|\alpha|}{2}+1}|D^{\beta}\eps_{t,\mu}|][\sum_{|\beta|\leq |\alpha|+1}|D^{\beta}_{\eps_{t,\mu}}|])
\end{align*}
Undoing the null-form expansion as on the right hand side of \eqref{eq:nullform}, i. e. writing things as on the left hand side of \eqref{eq:fullmess}, we find
\begin{align*}
&r^{-\frac{1}{2}}\int_0^t\int_{|r-(t-\tilde{t})|}^{r+(t-\tilde{t})}\chi^3 \mu^{\frac{1}{2}}|D^{\alpha_1}\eps_{t,\mu}|H_{\alpha_2}\,d\mu d\tilde{t}\\
&\lesssim r^{-\frac{1}{2}}\int_0^t\int_{|r-(t-\tilde{t})|}^{r+(t-\tilde{t})}\chi^3\mu^{\frac{1}{2}}|D^{\alpha_1}\eps_{t,\mu}|O(\frac{1}{\mu})\sum_{|\beta|\leq |\alpha_2|+1}|D^\beta\eps_\mu|\,d\mu d\tilde{t}\\
&+ r^{-\frac{1}{2}}\int_0^t\int_{|r-(t-\tilde{t})|}^{r+(t-\tilde{t})}\chi^3\mu^{\frac{1}{2}}|D^{\alpha_1}\eps_{t,\mu}|[\sum_{|\beta|\leq |\frac{\alpha_2|}{2}+1}|D^\beta\eps_{t,\mu}|][\sum_{|\beta|\leq |\alpha_2|+1}|D^\beta\eps_{t,\mu}|]\,d\mu d\tilde{t}\\
\end{align*}
Using the Cauchy-Schwarz inequality with respect to $\mu$ and assuming (as we may since $|\alpha_1|+|\alpha_2|\leq \frac{N}{2}+2$) $|\alpha_1|<\frac{N}{2}+2$, the first term is bounded by
\begin{align*}
&\lesssim \int_0^t \tilde{t}^{-\frac{\delta_1}{2}}(\lambda - \tilde{t})^{-1}\|\chi^3D^{\alpha_1}\eps_{t,\mu}\|_{L^\infty_{\mu\,d\mu}}\|\sum_{|\beta|\leq |\alpha_2|+1}|D^\beta\eps_\mu|\|_{L^2_{\mu\,d\mu}}\,d\tilde{t}\\
&\ll (K\kappa_0) \la t\ra^{-\frac{1}{2}}
\end{align*}
provided $\delta_1>2\delta$.
\\
To bound the second integral above, we have to use a different observation, namely that the restrictions $\mu\in [r+(t-\tilde{t}), |r-(t-\tilde{t})|]$, $|\tilde{t} - \mu|<\tilde{t}^{-\delta_1}$ imply that $\tilde{t}$ ranges over an interval of length $\sim r$. Then we find
\begin{align*}
&r^{-\frac{1}{2}}\int_0^t\int_{|r-(t-\tilde{t})|}^{r+(t-\tilde{t})}\chi^3\mu^{\frac{1}{2}}|D^{\alpha_1}\eps_{t,\mu}|[\sum_{|\beta|\leq |\frac{\alpha_2|}{2}+1}|D^\beta\eps_\mu|][\sum_{|\beta|\leq |\alpha_2|+1}|D^\beta\eps_\mu|]\,d\mu d\tilde{t}\\
&\lesssim (K\kappa_0)^3r^{\frac{1}{2}}\la t\ra^{\delta - \frac{\delta_1}{2}-1}\ll (K\kappa_0)\la t\ra^{-\frac{1}{2}}
\end{align*}
where we have estimated the factors $|D^{\alpha_1}\eps_{t,\mu}|, \sum_{|\beta|\leq |\frac{\alpha_2|}{2}+1}|D^\beta\eps_\mu|$ in $L^\infty_{\mu d\mu}$ and used Cauchy-Schwarz for the $\mu$-integral. 
This concludes case (ii.a).
\\
{\it{(iib): bounding the integral}}
\[
r^{-\frac{1}{2}}\sum_{\alpha_1+\alpha_2=\alpha}\big|\int_0^t\int_{|r-(t-\tilde{t})|}^{r+(t-\tilde{t})}\mu^{\frac{1}{2}}\frac{|D^{\alpha_1}\Gamma\eps|\, I_{\alpha_2}}{\tilde{t}^2 - \mu^2}\,d\mu d\tilde{t}\big|
\]
where
\[
I_\alpha = O(\sum_{|\beta|\leq |\alpha|}|D^{\beta}\Gamma_{1,2}\eps_{t,r}|\sum_{|\beta|\leq |\alpha|}|D^{\beta}\eps_{t,r}|)
\] 
and we have again reverted to writing things as on the right hand side in \eqref{eq:nullform}. 
By what was shown in (iia.2), (iia.2), it follows that it suffices to analyze the analogues (iib.1) - (iib.2) with the latter case only for $|\tilde{t}^2 - \mu^2|\geq \tilde{t}^{1-\delta_1}$.
\\
{\it{(iib.1): $\tilde{t}\ll t$.}} Again $\mu\gtrsim t$ on the support of the integrand. We estimate this (using Cauchy-Schwarz with respect to $\mu$) by
\[
\lesssim t^{-2}\la \log t\ra^{\frac{1}{2}}\int_0^t \|\frac{D^{\alpha_1}\Gamma\eps}{\la\log \mu\ra^{\frac{1}{2}}}\|_{L^\infty_{\mu\,d\mu}}\|I_{\alpha_2}\|_{L^2_{\mu\,d\mu}}\,d\tilde{t},\,\alpha_1+\alpha_2\leq \frac{N}{2}+2.
\]
Then we use
\[
\|\frac{D^{\alpha_1}\Gamma\eps}{\la\log \mu\ra^{\frac{1}{2}}}\|_{L^\infty_{\mu\,d\mu}}\lesssim \|D^{\alpha_1}\Gamma\eps\|_{\dot{H}^1_{\mu\,d\mu}}\lesssim (K\kappa_0)\la \tilde{t}\ra^{\delta}
\]
\[
\|I_{\alpha_2}\|_{L^2_{\mu\,d\mu}}\lesssim (K\kappa_0)^2 \la \tilde{t}\ra^{2\delta}
\]
and so we can bound
\[
t^{-2}\la \log t\ra^{\frac{1}{2}}\int_0^t \|D^{\alpha_1}\Gamma\eps\|_{L^\infty_{\mu\,d\mu}}\|I_{\alpha_2}\|_{L^2_{\mu\,d\mu}}\,d\tilde{t}\lesssim t^{-2}\la\log t\ra^{\frac{1}{2}}(K\kappa_0)^3\int_0^t\tilde{t}^{3\delta}\,d\tilde{t}\ll (K\kappa_0)t^{-\frac{1}{2}}
\]
provided $\delta$ is chosen small enough ($3\delta<\frac{1}{2}$).
\\

{\it{(iib.2): $\tilde{t}\gtrsim t$, $|\tilde{t}^2 - \mu^2|\geq \tilde{t}^{1-\delta_1}$.}} In case $\mu\ll \tilde{t}$, we can estimate this contribution just like in the preceding case (iib.1). Thus assume now $\mu\gtrsim \tilde{t}\gtrsim t$.
Due to Huyghen's principle, it suffices to restrict the integrand to the backward light cone centered at $(r, t)$. It then follows that we get the bound (on that portion of the integrand)
\[
|\mu^{\frac{1}{2}}D^{\alpha_1}\Gamma\eps|\lesssim r^{\frac{1}{2}}\|\partial_\mu D^{\alpha_1}\Gamma\eps\|_{L^2_{\mu\,d\mu}}
\]
In fact, we can write for $D^{\alpha_1}\Gamma\eps$ with arguments $(\tilde{t}, \mu)$ in the backward light cone centered at $(r, t)$ 
\[
\mu^{\frac{1}{2}}D^{\alpha_1}\Gamma\eps = \int_{\max\{\lambda-\tilde{t}, r-(t-\tilde{t}\}}^{\mu}\partial_{\mu}\big(\mu^{\frac{1}{2}}D^{\alpha_1}\Gamma\eps\big)\,d\mu
\]
and from Cauchy-Schwarz
\[
|D^{\alpha_1}\Gamma\eps|\leq r^{\frac{1}{2}}\|\partial_{\mu}D^{\alpha_1}\Gamma\eps\|_{L^2_{\mu d\mu}}
\]
Thus we infer the bound\footnote{The extra $\tilde{t}^{-1}$ comes from the weight in $\|\cdot\|_{L^1_{\mu\,d\mu}}$.} (here $\chi^4$ localizes to the region specified above)
\begin{align*}
&r^{-\frac{1}{2}}\int_0^t\int_{|r-(t-\tilde{t})|}^{r+(t-\tilde{t})}\chi^4\frac{|\mu^{\frac{1}{2}}D^{\alpha_1}\Gamma\eps|I_{\alpha_2}}{|\tilde{t}^2 - \mu^2|}\,d\mu d\tilde{t}\\
&\lesssim  \int_0^t \chi_{\tilde{t}\gtrsim t}\tilde{t}^{\delta_1 - 2}\|\partial_\mu D^{\alpha_1}\Gamma\eps\|_{L^2_{\mu\,d\mu}}\|I_{\alpha_2}\|_{L^1_{\mu\,d\mu}}\,d\tilde{t}\\
&\lesssim  (K\kappa_0)^3\int_0^t \chi_{\tilde{t}\gtrsim t} \tilde{t}^{\delta_1 + 3\delta - 2}\,d\tilde{t}
\ll (K\kappa_0)t^{-\frac{1}{2}}
\end{align*}
provided $\delta_1 + 3\delta<\frac{1}{2}$. This completes case (ii.b) and hence the contribution of the null-form.
\\

{\it{(iii): The contribution  of the source terms $F_\alpha(\tilde{t}, \mu)$; remaining terms.}} Here we explain how to deal with the two terms of \eqref{eq:enerror1}, the remaining ones being treated similarly.
\\
{\it{(iiia)}} We commence by bounding the expression ($|\alpha_2| + |\alpha_3|\leq \frac{N}{2}+2$)
\begin{align*}
r^{-\frac{1}{2}}\int_0^t\int_{|r-(t-\tilde{t})|}^{r+(t-\tilde{t})} \mu^{\frac{1}{2}}D^{\alpha_2}\big(O(\frac{1}{\mu^2})\big) D^{\alpha_3}\big(\frac{\eps_{\mu\mu} + O(\frac{1}{\mu})\eps_\mu}{(1+|\nabla_x\phi|^2 - \eps_t^2)^{\frac{3}{2}}}\big)\,d\mu d\tilde{t}
\end{align*}
We distinguish between the following cases:
\\

{\it{(iiia.1): $\tilde{t}\ll t$}}. In this case, we have $\mu\gtrsim t$ on the support of the integrand, and so we can bound this contribution by
\[
\lesssim t^{-1}\int_0^t (\lambda  - \tilde{t})^{-1}\|\sum_{|\beta|\leq |\alpha_3|+2}|D^{\beta}\eps|\|_{L^2_{\mu\,d\mu}}\,d\tilde{t}\ll (K\kappa_0)t^{-\frac{1}{2}},\,0\leq t\leq \lambda - C_1
\]
 if we use the already bootstrapped energy bounds.
\\

{\it{(iiia.2): $\tilde{t}\gtrsim t$}}. Here we write
\[
\eps_{\mu\mu} = \frac{1}{\tilde{t}}[\Gamma_1\eps - \mu\eps_{\tilde{t}}]_{\mu}
\]
Then we get (for $|\alpha_2|+|\alpha_3|\leq \frac{N}{2}+2$)
\begin{align*}
&r^{-\frac{1}{2}}\big|\int_0^t\int_{|r-(t-\tilde{t})|}^{r+(t-\tilde{t})} \chi_{\tilde{t}\gtrsim t}\mu^{\frac{1}{2}}D^{\alpha_2}\big(O(\frac{1}{\mu^2})\big) D^{\alpha_3}\big(\frac{\eps_{\mu\mu} + O(\frac{1}{\mu})\eps_\mu}{(1+|\nabla_x\phi|^2 - \eps_t^2)^{\frac{3}{2}}}\big)\,d\mu d\tilde{t}\big|\\
&\lesssim t^{-1}\int_0^t (\lambda  - \tilde{t})^{-2}\big[\sum_{1\leq|\beta|\leq |\alpha_3|+1}\|D^{\beta}\Gamma_1\eps\|_{L^2_{\mu\,d\mu}}+(\lambda - \tilde{t})\sum_{1\leq|\beta|\leq |\alpha_3|+2}\|D^{\beta}\eps\|_{L^2_{\mu\,d\mu}}\big]\,d\tilde{t}\\
&+t^{-1}\int_0^t (\lambda  - \tilde{t})^{-2}\big[\sum_{|\beta|\leq |\alpha_3|}\|\frac{D^{\beta}\Gamma_1\eps}{\mu}\|_{L^2_{\mu\,d\mu}}+(\lambda - \tilde{t})\sum_{|\beta|\leq |\alpha_3|+1}\|\frac{D^{\beta}\eps}{\mu}\|_{L^2_{\mu\,d\mu}}\big]\,d\tilde{t}\\
&\ll (K\kappa_0)t^{-1}\la\log t\ra t^{\delta}\lesssim (K\kappa_0)t^{-\frac{1}{2}},\,0\leq t\leq \lambda - C_1,
\end{align*}
where in the last step we use the already bootstrapped energy bounds.
\\

{\it{(iiib)}} Finally, we also treat the contribution of the term
\begin{align*}
r^{-\frac{1}{2}}\int_0^t\int_{|r-(t-\tilde{t})|}^{r+(t-\tilde{t})} \mu^{\frac{1}{2}}D^{\alpha_2}\big(O(\frac{1}{\mu})\big)D^{\alpha_3}\big(\frac{\eps_\mu\eps_{\mu\mu} - \eps_t\eps_{t\mu}}{(1+|\nabla_x\phi|^2 - \eps_t^2)^{\frac{3}{2}}}\big)\,d\mu d\tilde{t}
\end{align*}
We use the same case distinction as above:
\\
{\it{(iiib.1): $\tilde{t}\ll t$.}} We have $\mu\gtrsim t$. Here again we have to use the null-structure \eqref{eq:nullform}, i. e. write
\[
\big|D^{\alpha_3}\big(\frac{\eps_\mu\eps_{\mu\mu} - \eps_t\eps_{t\mu}}{(1+|\nabla_x\phi|^2 - \eps_t^2)^{\frac{3}{2}}}\big)\big|\lesssim \sum_{|\beta_1|+|\beta_2|\leq |\alpha_3|}\frac{D^{\beta_1}\Gamma \eps D^{\beta_2}\Gamma\eps_{\mu}}{|\mu^2 - \tilde{t}^2|}
\]
Then we can bound the above integral expression by
\begin{align*}
&\lesssim t^{-3}\sum_{|\beta_1|+|\beta_2|\leq \frac{N}{2}+2}\int_0^t \|D^{\beta_1}\Gamma \eps\|_{L^\infty_{\mu\,d\mu}}\|D^{\beta_2}\Gamma\eps_{\mu}\|_{L^2_{\mu\,d\mu}}\,d\tilde{t}\lesssim (K\kappa_0)^2 t^{-2+2\delta}\la \log t\ra^{\frac{1}{2}}\\&\ll (K\kappa_0)t^{-\frac{1}{2}}
\end{align*}
{\it{(iiib.2): $\tilde{t}\gtrsim t$.}} If we further restrict to $|\mu^2 - \tilde{t}^2|\gtrsim \tilde{t}^2$, we can argue exactly as before (one only gains $t^{-1+2\delta}$, which is enough). Thus assume now $|\mu^2 - \tilde{t}^2|\ll\tilde{t}^2$, whence $\mu\sim t$. We further distinguish between $|\mu^2 - \tilde{t}^2|\geq \tilde{t}^{1-\delta_1}$ and $|\mu^2 - \tilde{t}^2|<\tilde{t}^{1-\delta_1}$
In the former case, one infers for the integral above the bound
\[
\lesssim (K\kappa_0)^2 t^{-1+2\delta+\delta_1}\ll (K\kappa_0) t^{-\frac{1}{2}}
\]
provided $\delta, \delta_1$ are small enough. On the other hand, if we restrict to $\tilde{t}\gtrsim t,  |\mu^2 - \tilde{t}^2|<\tilde{t}^{1-\delta_1}$ via a cutoff $\chi^3$, we obtain without using the null-structure (and restricting to $|\alpha_2|+|\alpha_3|\leq \frac{N}{2}+2$)
\begin{align*}
&r^{-\frac{1}{2}}\big|\int_0^t\int_{|r-(t-\tilde{t})|}^{r+(t-\tilde{t})} \chi^3\mu^{\frac{1}{2}}D^{\alpha_2}\big(O(\frac{1}{\mu})\big)D^{\alpha_3}\big(\frac{\eps_\mu\eps_{\mu\mu} - \eps_t\eps_{t\mu}}{(1+|\nabla_x\phi|^2 - \eps_t^2)^{\frac{3}{2}}}\big)\,d\mu d\tilde{t}\big|\\
&\lesssim t^{-1}\int_0^t \tilde{t}^{-\frac{\delta_1}{2}}[\sum_{|\beta|\leq \frac{N}{2}}\|D^{\beta}\eps_{\mu,t}\|_{L^\infty_{\mu\,d\mu}}]\sum_{|\beta|\leq |\alpha_3|+1}\|D^{\beta}\eps_{\mu,t}\|_{L^2_{\mu\,d\mu}}]\,d\tilde{t}\\
&\lesssim (K\kappa_0)^2t^{-1}\int_0^t\tilde{t}^{\delta-\frac{\delta_1}{2}-\frac{1}{2}}\,d\tilde{t}\ll (K\kappa_0)t^{-\frac{1}{2}}
\end{align*}
provided $\delta_1>2\delta$; we have used Cauchy-Schwarz with respect to $\mu$. This completes estimating the contribution of the second term of \eqref{eq:enerror1}, and the remaining contributions from \eqref{eq:enerror2}, \eqref{eq:enerror3} are handled similarly.

\end{proof}

\end{proof}


\begin{thebibliography}{10}

\bibitem{Bren}Brendle, S., {\em Hypersurfaces in Minkowski space with vanishing mean curvature} CPAM 55(2002), no. 10, 1249 - 1279

\bibitem{Chr}Christodoulou, D., {\em Global solutions of nonlinear hyperbolic equations for small initial data} CPAM 39(1986), 267 - 282

\bibitem{Ho} Hoppe, J., {\em Some classical solutions of relativistic membrane equations in 4 space-time dimensions} Phys. Lett. B 329(1994), 10-14

\bibitem{H} H{\"o}rmander, L., {\em Lectures on Nonlinear Hyperbolic Differential Equations} Springer Verlag

\bibitem{J} John, F., {\em Partial Differential Equations} Springer Verlag

\bibitem{Kl}Klainerman, S., {\em Uniform decay estimates and the Lorentz invariance of the classical wave equations}, CPAM 38(1985), 321 - 332

\bibitem{Kl}Klainerman, S., {\em The Null Condition and global existence to nonlinear wave equations}, Lect. in Appl. Math. 23(1986), 293 - 326

\bibitem{Lin}Lindblad, H., {\em A remark on global existence for small initial data of the Minimal surface equation in Minkowskian space time} Proc. AMS 132(2004), 1095 - 1102

\bibitem{SchF} Fischer-Colbrie, D., Schoen, R., {\em The structure of complete stable minimal surfaces in $3$-manifolds of nonnegative scalar curvature} CPAM 33(1980), no.2, 199 - 211

\end{thebibliography}
\end{document}